\documentclass[a4paper,12pt]{article}
\usepackage[french, ngerman, italian, english]{babel}
\usepackage{amsmath,amsthm, amssymb}
\usepackage{setspace}
\usepackage{anysize}
\usepackage{url}
\usepackage{graphicx} 
\usepackage{hyperref}
\usepackage{mathptmx} 
\usepackage{paralist}
\usepackage{verbatim}
\usepackage[utf8]{inputenc}

\theoremstyle{plain}
\newtheorem{thm}{\bf Theorem}[section] 
\newtheorem{prop}[thm]{\bf Theorem}
\newtheorem{proposition}[thm]{\bf Proposition}
\newtheorem{thmnonumber}{\bf Main Theorem}
\newtheorem{lemma}[thm]{\bf Lemma}

\newtheorem{cor}[thm]{\bf Corollary}
\newtheorem*{cor*}{\bf Corollary}
\theoremstyle{definition}
\newtheorem{definition}[thm]{\bf Definition}
\theoremstyle{remark}
\newtheorem{remark}[thm]{\bf Remark}
\newtheorem{example}[thm]{\bf Example}

\begin{document}
\title{On locally constructible spheres and balls}
\author{Bruno Benedetti\thanks{%
Supported by DFG via the Berlin Mathematical School} \\
\small Inst.\ Mathematics, MA 6-2\\
\small TU Berlin\\
\small D-10623 Berlin, Germany\\
\small \url{benedetti@math.tu-berlin.de} 
\and \setcounter{footnote}{6} G\"unter M. Ziegler%
\thanks{%
Partially supported by DFG} \\ 
\small Inst.\ Mathematics, MA 6-2\\
\small TU Berlin\\
\small D-10623 Berlin, Germany\\
\small \url{ziegler@math.tu-berlin.de}}
\date{{\small June 18, 2009; revised July 21, 2010}} 
\maketitle

\begin{abstract}\noindent 
Durhuus and Jonsson (1995) 
introduced the class of ``locally constructible'' (LC) $3$-spheres
and showed that there are only exponentially-many combinatorial types of simplicial LC $3$-spheres.
Such upper bounds are crucial for the convergence of models for 3D quantum gravity.

We characterize the LC property for $d$-spheres (``the sphere minus
a facet collapses to a $(d-2)$-complex'') and for $d$-balls.  
In particular, we link it to the classical notions of collapsibility,
shellability and constructibility, and
obtain hierarchies of such properties for simplicial
balls and spheres. The main corollaries from this study are:
\begin{compactitem}[--]
\item Not all simplicial $3$-spheres are locally constructible.\\
(This solves a problem by Durhuus and Jonsson.)
\item There are only exponentially many shellable simplicial $3$-spheres with given number of facets. 
(This answers a question by Kalai.)
\item All simplicial constructible $3$-balls are collapsible.\\
(This answers a question by Hachimori.)
\item Not every collapsible $3$-ball collapses onto its boundary minus a facet.\\
(This property appears in papers by Chillingworth and Lickorish.)
\end{compactitem}
\end{abstract}  
 
\section{Introduction}

Ambj\o{}rn, Boulatov, Durhuus, Jonsson, and others 
have worked to develop a three-dimensional analogue of the
simplicial quantum gravity theory, as provided for two
dimensions by Regge \cite{REGGE}. (See \cite{ADJ} and \cite{REGGE1}
for surveys.)  The discretized version of
quantum gravity considers simplicial complexes instead of smooth
manifolds; the metric properties are artificially introduced by
assigning length $a$ to any edge. (This approach is due to Weingarten
\cite{Weingarten} and known as ``theory of dynamical
triangulations''.) A crucial path integral over metrics, the
``partition function for gravity'', is then defined via a weighted sum
over all triangulated manifolds of fixed topology. In three
dimensions, the whole model is convergent only if the number of
triangulated $3$-spheres with $N$ facets grows not faster than $C^N$,
for some constant~$C$. But does this hold? How many simplicial spheres are
there with $N$ facets, for $N$ large?

Without the restriction to ``local constructibility'' this crucial question still represents a
major open problem, which was put into the spotlight also by Gromov \cite[pp.~156-157]{Gromov}. Its
2D-analogue, however, was answered long time ago by Tutte \cite{TUT', TUT},
who proved that there are asymptotically fewer than
$\big(\frac{16}{3\sqrt{3}}\big)^N$ combinatorial types of
triangulated $2$-spheres.  (By Steinitz' theorem, 
cf.~\cite[Lect.~4]{Z}, this quantity equivalently counts the maximal
planar maps on \mbox{$n\ge4$} vertices, which have $N=2n-4$ faces, and also the
combinatorial types of simplicial $3$-dimensional polytopes with $N$
facets.)

In the following, the adjective ``simplicial'' will often be omitted
when dealing with balls, spheres, or manifolds, as all the
regular cell complexes and polyhedral complexes that 
we consider are simplicial.

Why are $2$-spheres ``not so many''? Every combinatorial type of
triangulation of the $2$-sphere can be generated as follows 
(Figure \ref{fig:treetriangles}): First for
some even $N\ge4$ build a tree of $N$ triangles (which combinatorially
is the same thing as a triangulation of an ($N+2$)-gon), and then glue
edges according to a complete matching of the boundary edges. A
necessary condition in order to obtain a $2$-sphere is that such a
matching is \emph{planar}.  Planar matchings and triangulations of
($N+2$)-gons are both enumerated by a Catalan number $C_{N+2}$, and
since the Catalan numbers satisfy a polynomial bound
$C_N=\frac1{N+1}\binom{2N}N<4^N$, we get an exponential upper bound for the
number of triangulations.
\vspace{-3mm}{}\enlargethispage{3mm}

\begin{figure}[htbf]
  \centering
  \includegraphics[width=.2\linewidth]{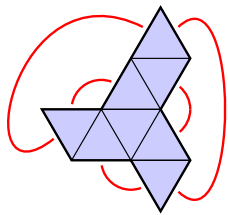}\vskip-4mm
  \caption{\small How to get an octahedron from a tree of $8$ triangles (i.e., a triangulated $10$-gon).}
  \label{fig:treetriangles}
\end{figure}

Neither this simple argument nor Tutte's precise count can be easily
extended to higher dimensions.  Indeed, we have to deal with three
different problems when trying to extend results or methods from dimension two to dimension
three:
\begin{compactenum}[(i)]
\item Many combinatorial types of simplicial $3$-spheres are not realizable
     as boundaries of convex $4$-polytopes; thus, even though we observe
     below that there are only exponentially-many simplicial
     $4$-polytopes with $N$ facets, the $3$-spheres could still be more numerous. 

\item The counts of combinatorial types according to the number
   $n$ of vertices and according to the number $N$ of facets
   are not equivalent any more. We have $3n-10\le N\le \frac{1}{2}n(n-3)$
   by the lower resp.\ upper bound theorem for simplicial
   $3$-spheres.
   We know that there are more than $2^{n\sqrt[4]{n}}$ $3$-spheres
   \cite{Kalai, PZ}, but less than $2^{20 n \log n}$ types of
   $4$-polytopes with $n$ vertices \cite{Alon, GP}, yet this
   does not answer the question for a count in terms of the number $N$ of
   facets.

\item While it is still true that there are only exponentially-many 
   ``trees of $N$ tetrahedra'', the matchings that can be used to glue
   $3$-spheres are not planar any more; thus, they could be
   more than exponentially-many.
   If, on the other hand, we restrict ourselves to ``local gluings'', 
   we generate only a limited family of $3$-spheres, as we will show below.
\end{compactenum}
In the early nineties, new finiteness theorems by Cheeger \cite{Cheeger} and
Grove et al.\ \cite{GPW} yielded a new approach, namely, to count $d$-manifolds 
of ``fluctuating topology'' (not necessarily spheres) but ``bounded geometry'' 
(curvature and diameter bounded from above, and volume bounded from below). 
This allowed Bartocci et al.\ \cite{Bartocci} to bound for any $d$-manifold the number of triangulations 
with $N$ or more facets, under the assumption that no vertex had 
degree higher than a fixed integer. However, for this it is crucial to
restrict the topological type:  
Already for $d=2$, there are more than exponentially many triangulated $2$-manifolds
of bounded vertex degree with $N$ facets.  

In 1995, the physicists Durhuus and Jonsson \cite{DJ} introduced the 
class of ``locally constructible'' (LC)
$3$-spheres. An LC $3$-sphere (with $N$ facets) is a sphere obtainable from a
tree of $N$ tetrahedra, by identifying pairs of adjacent triangles in
the boundary. ``Adjacent'' means here ``sharing at least one edge'',
and represents a dynamic requirement. Clearly, every $3$-sphere is
obtainable from a tree of $N$ tetrahedra by matching the triangles in
its boundary; according to the definition of LC, however, we are
allowed to match only those triangles that \emph{are} adjacent -- or
that have \emph{become} adjacent by the time of the gluing.

Durhuus and Jonsson proved an exponential upper bound
on the number of combinatorially distinct LC spheres with
$N$ facets. Based also on computer simulations
(\cite{AmbVar}, see also \cite{CatKR} and \cite{ABKW}) they
conjectured that all $3$-spheres should be LC. A positive solution of
this conjecture would have implied that spheres with $N$ facets are at
most $C^N$, for a constant $C$ -- which would have been the desired
missing link to implement discrete quantum gravity in three
dimensions. 

In the present paper, we show that the conjecture of Durhuus and
Jonsson has a negative answer: There are simplicial $3$-spheres that are
not LC. (With this, however, we do not resolve the
question whether there are fewer than $C^N$ simplicial $3$-spheres on
$N$ facets, for some constant~$C$.)

On the way to this result, we provide a characterization of LC simplicial $d$-complexes which relates the ``locally constructible'' spheres defined by physicists to concepts that originally arose in topological combinatorics.

\begin{thmnonumber} [Theorem~{\ref{thm:hierarchyspheres}}] \label{mainthm:hierarchyspheres}
 A simplicial $d$-sphere, $d\ge3$, is LC if and only if the sphere after removal of
 one facet can be collapsed down to a complex of dimension~$d-2$. Furthermore,
 there are the following inclusion relations between families of 
 simplicial $d$-spheres: \em
\[ 
 \{ \textrm{vertex decomposable} \} \subsetneq 
 \{ \textrm{shellable} \}  \subseteq
 \{ \textrm{constructible}\} \subsetneq 
 \{ \textrm{LC}\} \subsetneq
 \{ \textrm{all $d$-spheres}\}. 
\] 
\end{thmnonumber}

\noindent We use the hierarchy  in conjunction with the following
extension and sharpening of Durhuus and Jonsson's theorem (who
discussed only the case $d=3$).

\begin{thmnonumber} [Theorem~\ref{thm:announced}]
For fixed $d \geq 2$, the number of combinatorially distinct simplicial LC $d$-spheres
with $N$ facets grows not faster than $2^{ d^2 \cdot N }$.
\end{thmnonumber}

\noindent
We will give a proof for this theorem in Section~\ref{sec:numbers};
the same type of upper bound, with the same type of proof, also holds for LC $d$-balls with $N$ facets.

Already in 1988 Kalai \cite{Kalai} constructed for every $d \ge 4$ a family of more than exponentially many $d$-spheres on $n$ vertices; Lee \cite{Lee6} later showed that all of Kalai's spheres are shellable. Combining this with Theorem ~\ref{thm:announced} and   Theorem~\ref{thm:hierarchyspheres}, we obtain the following asymptotic result:

\begin{cor*}
For fixed $d \ge 4$, the number of shellable simplicial
$d$-spheres grows more than exponentially with respect to the number $n$ of vertices, 
but only exponentially with respect to the number $N$ of facets.
\end{cor*} 

\noindent
The hierarchy of Main Theorem~\ref{mainthm:hierarchyspheres} 
is not quite complete: It is still not known whether constructible,
non-shellable $3$-spheres exist (see \cite{EH, KAMEI}). 
A shellable $3$-sphere that is not vertex-\allowbreak decom\-pos\-able
was found by Lockeberg in his 1977 Ph.D.\ work (reported in
\cite[p.~742]{KK}; see also \cite{Htocome}). Again, the $2$-dimensional
case is much simpler and completely solved: All $2$-spheres are vertex
decomposable (see \cite{PB}).  

In order to show that not all spheres are LC
we study in detail simplicial spheres with a ``knotted triangle'';
these are obtained by adding a cone over the boundary of
a ball with a knotted spanning edge  (as in Furch's 1924 paper
\cite{FUR}; see also Bing \cite{BING}).
Spheres with a knotted triangle cannot be boundaries of polytopes.  
Lickorish \cite{LICK} had shown in 1991 that
\begin{quote}
  \emph{a $3$-sphere with a knotted triangle is not shellable
  if the knot is at least $3$-complicated.}
\end{quote}
Here ``at least $3$-complicated'' refers to the technical requirement
that the fundamental group of the complement of the knot has
no presentation with less than four generators. A concatenation of
three or more trefoil knots satisfies this condition. 
In 2000, Hachimori and Ziegler \cite{HACHI, HZ}
demonstrated that Lickorish's technical requirement is not necessary
for his result:
\begin{quote}
  \emph{a $3$-sphere with {\em any} knotted triangle is not constructible.}
\end{quote}
In the present work, we re-justify Lickorish's technical assumption,
showing that this is {exactly} what we need if we want to reach a
stronger conclusion, namely, a topological obstruction to local
constructibility. Thus, the following result is established in
order to prove that the last inclusion of the hierarchy in 
Theorem~\ref{thm:hierarchyspheres} is strict.

\begin{thmnonumber}  [Theorem~\ref{thm:short}] 
  A $3$-sphere with a knotted
  triangle is not LC if the knot is at least $3$-complicated.
\end{thmnonumber}

The knot complexity requirement is now necessary, as non-constructible
spheres with a single trefoil knot can still be LC (see
Example~\ref{thm:examplelick}).

The combinatorial topology of $d$-balls and that of $d$-spheres
are of course closely related -- our study builds on the
well-known connections and also adds new ones.

\begin{thmnonumber} [Theorems~\ref{thm:hierarchyballs} and~\ref{thm:mainDballs}]
\label{mainthm:hierarchyballs}
A simplicial $d$-ball is LC if and only if after the removal of a facet it collapses down
to the union of the boundary with a complex of dimension at most $d-2$. 
We have the following hierarchy: \em 
\[ \Big\{\begin{array}{@{}c@{}}
	\textrm{vertex}\\ \textrm{decomp.}
    \end{array}\Big\} \subsetneq 
   \{\textrm{shellable}\} \subsetneq
   \{\textrm{constructible}\} \subsetneq  
   \{\textrm{LC}\} \subsetneq
   \Big\{\begin{array}{@{}c@{}}
   \textrm{collapsible onto a}\\
   (d-2)\textrm{-complex}
   \end{array}\Big\} \subsetneq
   \{\textrm{all $d$-balls}\}. 
\]
\end{thmnonumber}

All the inclusions of Main Theorem \ref{mainthm:hierarchyballs} hold with equality 
for simplicial $2$-balls. In the case of 
$d=3$, collapsibility onto a $(d-2)$-complex is equivalent to
collapsibility. 
In particular, we settle a question of Hachimori (see e.g.\ \cite[pp.~54, 66]{Hthesis})
whether all constructible $3$-balls are collapsible. 

Furthermore, we show in Corollary~\ref{cor:badcollapsible3ball} that some 
collapsible $3$-balls do not collapse onto their boundary minus a facet, a property that
comes up in classical studies in combinatorial topology (compare \cite{CHIL, LICK2}).
In particular, a result of Chillingworth can be restated in our language
as ``if for any geometric simplicial complex $\Delta$ the support (union) $|\Delta|$ is a
convex $3$-dimensional polytope, then $\Delta$ is necessarily an LC
$3$-ball'', see Theorem~\ref{thm:chil}. Thus any geometric subdivision of the $3$-simplex is LC.

\subsection{Definitions and Notations}

\subsubsection{Simplicial regular CW complexes}

In the following, we present the notion of ``local constructibility''
(due to Durhuus and Jonsson).
Although in the end we are interested in this notion as applied
to finite simplicial complexes, the iterative definition of 
locally constructible complexes dictates that for intermediate steps
we must allow for the greater generality of
finite ``simplicial regular CW complexes''.
A CW complex is \emph{regular} if the attaching maps for the cells are
injective on the boundary (see  e.g. \cite{BJOE}). A regular
CW-complex is \emph{simplicial} if for every proper face $F$, the
interval $[0, F]$ in the face poset of the complex is boolean. Every
simplicial complex (and in particular, any triangulated manifold) is a
simplicial regular CW-complex.

The $k$-dimensional cells of a regular CW complex $C$ are called
$k$-\emph{faces}; the inclusion-maximal faces are called
\emph{facets}, and the inclusion-maximal proper subfaces of the
facets are called \emph{ridges}.
The \emph{dimension} of $C$ is
the largest dimension of a facet; \emph{pure} complexes are
complexes where all facets have the same dimension. 
All complexes that we consider in the following are finite,
most of them are pure.
A \emph{$d$-complex} is a $d$-dimensional complex.
Conventionally,
the 0-faces are called \emph{vertices}, and the $1$-faces
\emph{edges}. (In the discrete quantum gravity literature, 
the $(d-2)$-faces are sometimes  called ``{hinges}'' or ``{bones}'', whereas the edges are 
sometimes referred to as ``{links}''.) 
If the union $|C|$ of all simplices of $C$ is
homeomorphic to a manifold $M$, then $C$ is a \emph{triangulation}
of $M$; if $C$ is a triangulation of a $d$-ball or of a $d$-sphere, we
will call $C$ simply a $d$-\emph{ball} (resp. $d$-\emph{sphere}). 
The \emph{dual graph} of a pure $d$-dimensional simplicial complex~$C$ is the graph whose nodes 
correspond to the facets of~$C$: Two nodes are connected by an arc if and only if the corresponding facets share a $(d-1)$-face.

\subsubsection{Knots}
All the knots
we consider are \textit{tame}, that is, realizable as $1$-dimensional
subcomplexes of some triangulated $3$-sphere. A knot is
$m$-\emph{complicated} if the fundamental group of the complement of
the knot in the $3$-sphere has a presentation with $m+1$ generators, but
no presentation with $m$ generators. By ``at least $m$-complicated''
we mean ``$k$-complicated for some $k\geq m$''. There exist
arbitrarily complicated knots: Goodrick \cite{GOO} showed that
the connected sum of $m$ trefoil knots is at least $m$-complicated. 

Another measure of how tangled a knot can be is the  {bridge index} 
(see e.g.\ \cite[p.~18]{KAWA} for the definition). 
If a knot has bridge index $b$, the fundamental group of the knot
complement admits a presentation with $b$ generators and $b-1$
relations \cite[p.~82]{KAWA}. 
In other words, the bridge index of a $t$-complicated knot is at least
$t+1$. As a matter of fact, the connected sum of $t$ trefoil knots is
$t$-complicated, and its bridge index is exactly $t+1$ \cite{EH}.

\subsubsection{The combinatorial topology hierarchy}

In the following, we review the key properties from the inclusion 
\[ \{\textrm{shellable}\} \subsetneq
   \{\textrm{constructible}\} 
\]
valid for all simplicial complexes, and the inclusion
\[ \{\textrm{shellable}\} \subsetneq
   \{\textrm{collapsible}\} 
\]
applicable only for \emph{contractible} simplicial complexes, both known from combinatorial
topology (see \ \cite[Sect.~11]{BJOE} for details).

\noindent
\emph{Shellability} can be defined for pure simplicial complexes as
follows:
\begin{compactitem}[ -- ]
\item every simplex is shellable; 
\item a $d$-dimensional pure simplicial complex $C$ which is not a
  simplex is shellable if and only if it can be written as $C= C_1
  \cup C_2$, where $C_1$ is a shellable
  $d$-complex, $C_2$ is a $d$-simplex, and $C_1 \cap C_2$ is a shellable $(d-1)$-complex.
\end{compactitem} 
\emph{Constructibility} is a weakening of shellability, defined by:
\begin{compactitem}[ -- ]
\item every simplex is constructible; 
\item a $d$-dimensional pure simplicial complex $C$ which is not a
  simplex is constructible if and only if it can be written as $C= C_1
  \cup C_2$, where $C_1$ and $C_2$ are constructible $d$-complexes,
  and $C_1 \cap C_2$ is a constructible $(d-1)$-complex.
\end{compactitem}
\smallskip \noindent
Let $C$ be a $d$-dimensional simplicial complex.  
An \emph{elementary collapse} is the simultaneous removal from $C$ of a
pair of faces $(\sigma, \Sigma)$ with the following prerogatives:
\begin{compactenum}[ -- ]
\item $\dim \Sigma = \dim \sigma + 1 $;
\item $\sigma$ is a proper face of $\Sigma$;
\item $\sigma$ is not a proper face of any other face of $C$.
\end{compactenum}
(The three conditions above are usually abbreviated in the expression 
``$\sigma$ is a free face of $\Sigma$''; some complexes have no free face).
If $C':= C - \Sigma - \sigma$, we say that the complex $C$
\emph{collapses onto} the complex $C'$.  We also say that the
complex $C$ \emph{collapses onto} the complex $D$, and write $C \searrow D$,
if $C$ can be reduced to $D$ by a finite sequence of elementary collapses. 
Thus a \emph{collapse} refers to a sequence of elementary collapses.
A \emph{collapsible} complex is a complex that can be collapsed onto a
single vertex. 

Since $C':= C - \Sigma - \sigma$ is a deformation retract of $C$, each
collapse preserves the homotopy type. In particular, all collapsible
complexes are contractible. The converse does not hold in general: 
For example, the so-called ``dunce hat''  
is a contractible $2$-complex
without free edges, and thus with no elementary collapse to start with.
However, the implication ``contractible $\Rightarrow$
collapsible'' holds for all $1$-complexes, and also for shellable
complexes of any dimension.

A connected $2$-dimensional complex is collapsible if and only
if it does \emph{not} contain a $2$-dimensional complex without a free
edge. In particular, for $2$-dimensional complexes, if $C \searrow D$
and $D$ is not collapsible, then $C$ is also not collapsible. This
holds no more for complexes $C$ of dimension larger than two \cite{HOGAM}. 

\subsubsection{LC pseudomanifolds}

By a $d$-\emph{pseudomanifold} [possibly with boundary] we mean a
finite regular CW-complex $P$ that is pure $d$-dimensional,
simplicial, and such that each $(d-1)$-dimensional cell belongs to at
most two $d$-cells. The \emph{boundary} of the pseudomanifold $P$,
denoted $\partial P$, is the smallest subcomplex of $P$ containing all
the $(d-1)$-cells of $P$ that belong to exactly one $d$-cell of $P$.

According to our definition, a pseudomanifold needs not be a
simplicial complex; it might be disconnected; and its boundary might not
be a pseudomanifold.

 \begin{definition}[Locally constructible pseudomanifold] 
 For $d\ge2$, let $C$ be a pure $d$-dimensional simplicial complex with
 $N$ facets. A \emph{local construction} for $C$ is a sequence $T_1,
 T_2, \ldots, T_N, \ldots, T_k$ ($k \geq N$) such that $T_i$ is a
 $d$-pseudomanifold for each $i$ and
\begin{compactenum}[(1)]
\item $T_1$ is a $d$-simplex; 
\item if $i \le N-1$, then $T_{i+1}$ is obtained from $T_i$
  by gluing a new $d$-simplex to $T_i$ alongside one of the
  $(d-1)$-cells in $\partial T_i$;
\item if $i  \ge N$, then $T_{i+1}$ is obtained from $T_{i}$
  by identifying a pair $\sigma, \tau$ of $(d-1)$-cells in the boundary $\partial T_{i}$ whose
  intersection contains a $(d-2)$-cell $F$; 
\item $T_k = C$.
\end{compactenum}
We say that $C$ is \emph{locally constructible}, or \emph{LC}, if a local construction for $C$ exists.
With a little abuse of notation, we will call each $T_i$ an
\emph{LC pseudomanifold}. 
We also say that $C$ is locally constructed \emph{along $T$}, 
if $T$ is the dual graph of $T_N$, and thus a spanning tree of
the dual graph of~$C$.
\end{definition}

The identifications described in item (3) above are
operations that   
are not closed with respect to the class of simplicial
complexes. Local constructions where all steps are simplicial complexes
produce only a very limited class of manifolds,  
consisting of $d$-balls with no interior $(d-3)$-faces. 
(When in an LC step the identified boundary facets intersect in \emph{exactly} a $(d-2)$-cell, no $(d-3)$-face is sunk into the interior, and the topology stays the same.)
 
However, since by definition the local construction in the end must arrive at a
pseudomanifold $C$ that \textit{is} a simplicial complex, each intermediate
step $T_i$ must satisfy severe restrictions: for each $t \leq d$, 
\begin{compactitem}[ -- ]
\item distinct $t$-simplices that are not in the boundary of $T_i$
  share at most one $(t-1)$-simplex;
\item distinct $t$-simplices in the boundary of $T_i$ that share more
  than one $(t-1)$-simplex will need to be identified by the time the
  construction of $C$ is completed.
\end{compactitem}
Moreover,
\begin{compactitem}[ -- ]
 \item if $\sigma, \tau$ are the two $(d-1)$-cells glued together in the step from $T_i$ to $T_{i+1}$, $\sigma$ and $\tau$ cannot belong to the same $d$-simplex of $T_i$; nor can they belong to two $d$-simplices that are already adjacent in $T_i$.
\end{compactitem}
For example, in each step of the local construction of a $3$-sphere, no
two tetrahedra share more than one triangle. Moreover, any two distinct interior triangles 
either are disjoint, or they share a vertex, or they share an edge; but they cannot share two edges, nor three; and they also cannot share one edge and the opposite vertex.
If we glued together two boundary triangles that belong to adjacent tetrahedra, no matter what we did
afterwards, we would not end up with a simplicial complex any more.
Roughly speaking, 
\begin{quote}\emph{a locally constructible $3$-sphere is a triangulated $3$-sphere
obtained from a tree of tetrahedra $T_N$ by repeatedly
identifying two adjacent triangles in the boundary.}
\end{quote} 
As we mentioned, the boundary of a pseudomanifold need not be a
pseudomanifold. However, if $P$ is an LC $d$-pseudomanifold, then
$\partial P$ is automatically a $(d-1)$-pseudomanifold. Nevertheless, 
$\partial P$ may be disconnected, and thus, in general, it is not LC. 

All LC $d$-pseudomanifolds are simply connected; in case $d=3$, their topology 
is controlled by the following result.

\begin{thm}[Durhuus--Jonsson {\cite{DJ}}] \label{thm:topologyLC3-dim}
Every LC $3$-pseudomanifold $P$ is homeomorphic to a $3$-sphere with a finite 
number of ``cacti of $3$-balls'' removed.
(A cactus of $3$-balls is a tree-like connected structure in which any two $3$-balls share at most one point.)
Thus the boundary $\partial P$ is a finite 
disjoint union of cacti of $2$-spheres. In particular, each connected component of 
$\partial P$ is a simply-connected $2$-pseudomanifold.
\end{thm}

Thus every closed $3$-dimensional LC pseudomanifold is a sphere, 
while for $d>3$ other topological types such as products of spheres are 
possible (see Benedetti \cite{Benedetti-JCTA}).

\section{On LC Spheres}

In this section, we establish the following hierarchy
announced in the introduction.

\begin{thm} \label{thm:hierarchyspheres}
	For all $d\ge3$, we have the following inclusion relations between families of simplicial $d$-spheres: \em
\[ 
 \{ \textrm{vertex decomposable} \} \subsetneq 
 \{ \textrm{shellable} \}  \subseteq 
 \{ \textrm{constructible}\} \subsetneq 
 \{ \textrm{LC}\} \subsetneq
 \{ \textrm{all $d$-spheres}\}. 
\]
\end{thm}

\begin{proof}
The first two inclusions, and strictness of the first one, are known; the third one will follow from 
Lemma~\ref{lem:LCdecomposition} and will be shown to be strict by Example \ref{thm:examplelick}
together with Lemma~\ref{thm:suspensions}; finally, 
Corollary \ref{thm:cor3} will establish the strictness of the fourth inclusion for all  $d\ge3$.%
\end{proof}

\subsection{Some $d$-spheres are not LC}\label{sec:notLCspheres}

Let $S$ be a simplicial $d$-sphere ($d \geq 2$), and $T$ a spanning tree of the dual graph
of $S$. We denote by $K^T$ the subcomplex of $S$ formed by all the
$(d-1)$-faces of $S$ that are not intersected by~$T$.%
 
\begin{lemma} \label{thm:Scollapse} Let $S$ be any $d$-sphere with $N$ facets. Then for every 
  spanning tree $T$ of the dual graph of $S$, 
\begin{compactitem}
\item $K^T$ is a contractible pure $(d-1)$-dimensional simplicial
  complex with $\frac{dN - N + 2}{2}$ facets;
\item for any facet $\Delta$ of $S$, $\; \;S - \Delta \; \searrow K^T$.
\end{compactitem}
\end{lemma}
\enlargethispage{3mm}

\noindent
Any collapse of a $d$-sphere $S$ minus a facet $\Delta$ to a complex of dimension at most $d-1$ 
 proceeds along a dual spanning tree~$T$.
To see this, fix a collapsing sequence. We may assume that the collapse of $S - \Delta$ is ordered so that the pairs 
$( (d-1)\textrm{-face} , \; d\textrm{-face})$
are removed first.  
Whenever both the following conditions are met:
\begin{compactenum}
\item $\sigma$ is the $(d-1)$-dimensional intersection of the facets
  $\Sigma$ and $\Sigma'$ of $S$;
\item the pair $(\sigma, \Sigma)$ is removed in the collapsing sequence of $S - \Delta$,
\end{compactenum}
draw an oriented arrow from the center of $\Sigma'$ to the center of $\Sigma$. 
This yields a directed spanning tree $T$ of the dual graph of
$S$, where $\Delta$ is the root. 
Indeed, $T$ is \emph{spanning} because all
$d$-simplices of $S - \Delta$ are removed in the collapse; it is \emph{connected}, because the only
free $(d-1)$-faces of $S-\Delta$, where the collapse can start at, are
the proper $(d-1)$-faces of the ``missing simplex'' $\Delta$; it is
\emph{acyclic}, because the center of each $d$-simplex of $S- \Delta$ is reached by exactly one arrow.  
We will say that the collapsing sequence \emph{acts along the tree} $T$ (in its top-dimensional part). 
Thus the complex $K^T$ appears as intermediate step 
of the collapse: It is the complex obtained after the $(N-1)$st pair of faces has been removed from $S - \Delta$.

\begin{definition} \label{thm:facetkilling}
By a \textit{facet-killing sequence} for a $d$-dimensional simplicial complex $C$ we mean a sequence $C_0, C_1, \ldots, C_{t-1}, C_t $ of complexes such that $t=f_d(C)$, $C_0=C$, and $C_{i+1}$ is obtained by an elementary collapse that removes a free $(d-1)$-face $\sigma$ of $C_i$, together with the unique facet $\Sigma$ containing $\sigma$.
\end{definition}

If $C$ is a $d$-complex, and $D$ is a lower-dimensional complex such that $C \searrow D$, there exists a facet-killing sequence $C_0$, $\ldots$, $C_t$ for $C$ such that $C_t \searrow D$. In other words, the collapse of $C$ onto $D$ can be rearranged so that the pairs $\left((d-1)\textrm{-face}, \, d\textrm{-face} \right)$ are removed first. 
In particular, for any $d$-complex $C$, the following are equivalent:
\begin{compactenum}
 \item there exists a facet-killing sequence for $C$;
\item there exists a $k$-complex $D$ with $k \leq d-1$ such that $C \searrow D$.
\end{compactenum}
What we argued before can be rephrased as follows:

\begin{proposition} \label{thm:Salongtrees}
Let $S$ be a $d$-sphere, and $\Delta$ a $d$-simplex of $S$. Let $C$ be a $k$-dimensional simplicial complex, with $k \leq d-2$. Then,
\[ S- \Delta \; \searrow \; C \ \ \Longleftrightarrow \ \ \exists \; T \hbox{ s.t. } \; K^T \; \searrow C. \]
\end{proposition}
 
\noindent
The right-hand side in the equivalence of Proposition \ref{thm:Salongtrees} does not 
depend on the $\Delta$ chosen. So, for any $d$-sphere $\Delta$, either $S - \Delta$ 
is collapsible for every $\Delta$, or $S - \Delta$ is not collapsible for any~$\Delta$.

\begin{figure}[htbf]
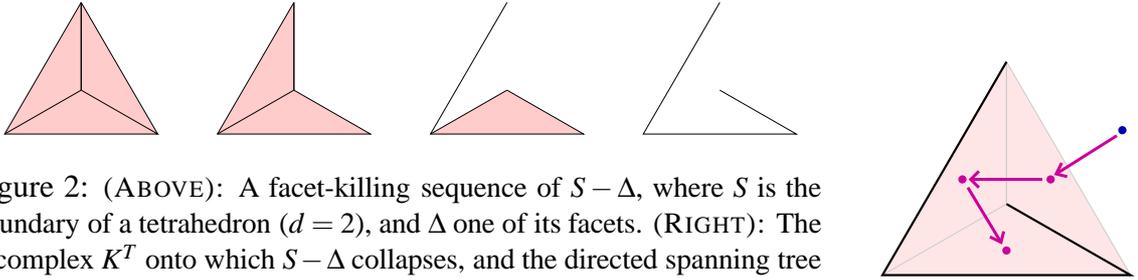

\begin{minipage}{0.7\linewidth}
 \hfill 
  \includegraphics[width=.2\linewidth]{collapse.eps} \hfill
  \includegraphics[width=.2\linewidth]{collapse1.eps} \hfill
  \includegraphics[width=.2\linewidth]{collapse2.eps} \hfill
  \includegraphics[width=.2\linewidth]{collapse3.eps} \hfill 
  \caption{\small \textsc{(Above):} A facet-killing sequence of $S - \Delta$,
  where $S$ is the boundary of a tetrahedron ($d=2$), and $\Delta$ one of its facets. 
  \small  \textsc{(Right):} The $1$-complex $K^T$  onto which $S - \Delta$ collapses, and the directed spanning tree $T$ 
   along which the collapse above acts.}
  \label{fig:collapsealongtrees}
\end{minipage}
\begin{minipage}{0.27\linewidth}
\hfill \includegraphics[width=.84\linewidth]{collapsetree.eps} 
\end{minipage}
\end{figure}

One more convention: by a \textit{natural labeling} of a rooted tree $T$ on $n$ vertices we mean a bijection $b : V(T) \longrightarrow \{1, \ldots, n\}$ such that if $v$ is the root, $b(v)=1$, and if $v$ is not the root, there exists a 
unique vertex $w$ adjacent to $v$ such that $b(w) < b(v)$.

\medskip
We are now ready to link the LC concept with collapsibility. 
Take a $d$-sphere $S$, a facet $\Delta$ of~$S$, and a rooted spanning tree $T$ 
of the dual graph of $S$, with root $\Delta$. 
Since $S$ is given, fixing $T$ is really the same as fixing the manifold $T_N$ 
in the local construction of $S$; and at the same time, fixing $T$ is the same as fixing $K^T$. 

Once $T$, $T_N$, and $K^T$ have been fixed, to describe the first part of a 
local construction of $S$ (that is, $T_1, \ldots, T_N$) we just need to specify 
the order in which the tetrahedra of $S$ have to be added, which is the same as to 
give a natural labeling of $T$. Besides, natural labelings of $T$ are in bijection
with collapses $S-\Delta \searrow K^T$ (the $i$-th facet to be collapsed is the node 
of $T$ labeled~$i+1$; see Proposition~\ref{thm:Salongtrees}).

What if we do not fix $T$? Suppose $S$ and $\Delta$ are fixed. Then the previous reasoning yields a bijection among the following sets:
\begin{compactenum}
\item the set of all facet-killing sequences of $S-\Delta$;
 \item the set of ``natural labelings'' of spanning trees of $S$, rooted at $\Delta$;
 \item the set of the first parts $(T_1, \ldots, T_N )$ of local constructions for $S$, 
with $T_1 = \Delta$.
\end{compactenum}

\noindent
Can we understand also the second part of a local construction ``combinatorially''? 
Let us start with a variant of the ``facet-killing sequence'' notion.

\begin{definition} \label{thm:facetmassacre}
A \textit{pure facet-massacre} of a pure $d$-dimensional simplicial complex $P$ is a sequence $P_0, P_1, \ldots, P_{t-1}, P_t$
of (pure) complexes such that $t=f_d(P)$, $P_0=P$, and $P_{i+1}$ is obtained by $P_i$ removing:
	\begin{compactdesc}
	\item{(a)} a free $(d-1)$-face $\sigma$ of $P_i$, together with the unique facet $\Sigma$ containing $\sigma$, and
	\item{(b)} all inclusion-maximal faces of dimension smaller than $d$ that are left after the removal of type (a) or, recursively, after removals of type (b).
	\end{compactdesc}
\end{definition}

\noindent
In other words, the (b) step removes lower-dimensional facets until one obtains a pure complex. 
Since $t=f_d(P)$, $P_t$ has no facets of dimension $d$ left, nor inclusion-maximal faces of smaller dimension; hence $P_t$ is empty. The other $P_i$'s are pure complexes of dimension $d$. Notice that the step $P_i \longrightarrow P_{i+1}$ is not a collapse, and does not preserve the homotopy type in general. Of course $P_i \longrightarrow P_{i+1}$ can be ``factorized'' in an elementary collapse followed by a removal of a finite number of $k$-faces, with $k <d$. However, this factorization is not unique, as the next example shows.

\begin{example}
Let $P$ be a full triangle. $P$ admits three different facet-killing collapses (each edge can be chosen as free face), but it admits only one pure facet-massacre, namely $P, \emptyset$.
\end{example}

\begin{lemma}
Let $P$ be a pure $d$-dimensional simplicial complex. Every facet-killing sequence of $P$ 
naturally induces a unique pure facet-massacre of $P$. All pure facet-massacres of $P$ are 
induced by some (possibly more than one) facet-killing sequence.
\end{lemma}

\begin{proof}
The map consists in taking a facet-killing sequence $C_0$, $\ldots$, $C_t$, and ``cleaning up'' the $C_i$ by recursively killing the lower-dimensional inclusion-maximal faces. As the previous example shows, this map is not injective. It is surjective essentially because the removed lower-dimensional faces are of dimension ``too small to be relevant''. In fact, their dimension is at most $d-1$, hence their presence can interfere only with the freeness of faces of dimension at most $d-2$; so the list of all removals of the form $( (d-1)\hbox{-face}, \,   d\hbox{-face} )$ in a facet-massacre yields a facet-killing sequence.
\end{proof}

\begin{thm} \label{thm:Smassacre}
Let $S$ be a $d$-sphere; fix a spanning tree $T$ of the dual graph of $S$. The second part of a local construction for $S$ along $T$ corresponds bijectively to a facet-massacre of $K^T$. 
\end{thm}

\begin{proof}
Fix $S$ and $T$; $T_N$ and $K^T$ are determined by this. 
Let us start with a local construction $\left( T_1, \ldots, T_{N-1}, \right) T_N, \ldots, T_k$ for $S$ along $T$.
Topologically, $S=T_N / {\sim}$, where ${\sim}$ is the equivalence relation determined by the gluing 
(two distinct points of $T_N$ are equivalent if and only if they will be identified in the gluing). 
Moreover, $K^T = \partial T_N / {\sim}$, by the definition of $K^T$.

Define $P_0 := K^T = \partial T_N / {\sim}$, and $P_j := \partial T_{N+j} / {\sim}$. We leave it to the reader to verify that 
$k-N$ and $f_d(K^T)$ are the same integer (see Lemma \ref{thm:Scollapse}), which 
from now on is called $D$. In particular $P_D=\partial T_k / {\sim} = \partial S / {\sim} = \emptyset$.

In the first LC step, $T_N \rightarrow T_{N+1}$, we remove from the boundary a free ridge $r$, together with the unique pair $\sigma', \sigma''$ of facets of $\partial T_N$ sharing $r$. At the same time, $r$ and the newly formed face $\sigma$ are sunk into the interior. This step $\partial T_N \longrightarrow \partial T_{N+1}$ naturally induces an analogous step $\partial T_{N+j}/ {\sim} \longrightarrow \partial T_{N+j+1}/ {\sim}$, namely, the removal of $r$ and of the (unique!) $(d-1)$-face $\sigma$ containing it. 

In the $j$-th LC step, $\partial T_{N+j} \longrightarrow \partial T_{N+j+1}$, we remove from the boundary a ridge $r$ together with a pair $\sigma', \sigma''$ of facets sharing $r$; moreover, we sink into the interior a lower-dimensional face $F$ if and only if we have just sunk into the interior all faces containing $F$. The induced step from $\partial T_{N+j}/ {\sim}$ to  $\partial T_{N+j+1}/ {\sim}$ is precisely a ``facet-massacre'' step.

For the converse, we start with a ``facet-massacre'' $P_0$, \ldots, $P_D$ of $K^T$,  
and
we have $P_0 = K_T = \partial T_N / {\sim}$. The unique $(d-1)$-face $\sigma_j$ killed in passing from $P_j$ to $P_{j+1}$ corresponds to a unique pair of (adjacent!) $(d-1)$-faces $\sigma_j'$, $\sigma_j''$ in $\partial T_{N+j}$. 
Gluing them together is the LC move that transforms $T_{N+j}$ into $T_{N+j+1}$. 
\end{proof}
\newpage

\begin{remark} \label{thm:newremark}
Summing up:
\begin{compactitem}[--]
\item The first part of a local construction along a tree $T$ corresponds to a facet-killing collapse of $S - \Delta$ (that ends in $K^T$).
\item The second part of a local construction along a tree $T$ corresponds to a pure facet-massacre of $K^T$.
\item A single facet-massacre of $K^T$ corresponds to many facet-killing sequences of $K^T$. 
\item By Proposition \ref{thm:Salongtrees}, there exists a facet-killing sequence of $K^T$ if and only if $K^T$ collapses onto some $(d-2)$-dimensional complex $C$. This $C$ is necessarily contractible, like $K^T$. 
\end{compactitem}
\end{remark}
\noindent
So $S$ is locally constructible along $T$ if and only if $K^T$ collapses onto some $(d-2)$-dimensional contractible complex $C$, if and only if $K^T$ has a facet-killing sequence.
What if we do not fix~$T$?

\begin{thm} \label{thm:Dcollapse} Let $S$ be a $d$-sphere ($d \ge 3$). 
Then the following are equivalent:
\begin{compactenum}[\rm 1.]
\item S is LC;
\item for some spanning tree $T$ of $S$, $K^T$ is collapsible 
onto some $(d-2)$-dimensional (contractible) complex $C$;
\item there exists a $(d-2)$-dimensional
  (contractible) complex $C$ such that for every facet $\Delta$ of $S$, $S - \Delta \searrow C$; 
\item for some facet $\Delta$ of $S$, $S - \Delta$ is collapsible onto a $(d-2)$-dimensional
  contractible complex~$C$.
\end{compactenum}
\end{thm}

\begin{proof}  
$S$ is LC if and only if it is LC along some tree $T$; thus $(1) \Leftrightarrow (2)$ follows from Remark \ref{thm:newremark}. Besides, $(2) \Rightarrow (3)$ follows from the fact that $ S  - \Delta \; \searrow K^T$ (Lemma \ref{thm:Scollapse}), where $K^T$ is independent of the choice of $\Delta$. $(3) \Rightarrow (4)$ is trivial. To show $(4) \Rightarrow (2)$, take a collapse of $S- \Delta$ onto some $(d-2)$-complex $C$; by Lemma \ref{thm:Salongtrees}, there exists some tree $T$ (along which the collapse acts) so that $S- \Delta \searrow K^T$ and $K^T \searrow C$.
\end{proof}

\begin{cor}\label{thm:corollarycollapse} Let $S$ be a $3$-sphere. Then the following are equivalent:
\begin{compactenum}[\rm 1.]
\item S is LC;
\item $K^T$ is collapsible, for some spanning tree $T$ of the dual graph of $S$;
\item $S - \Delta$ is collapsible for every facet $\Delta$ of $S$;
\item $S - \Delta$ is collapsible for some facet $\Delta$ of $S$.
\end{compactenum}
\end{cor}

\begin{proof}
This follows from the previous theorem, together with the fact that all contractible $1$-complexes are collapsible.
\end{proof}

We are now in the position to exploit results by Lickorish about
collapsibility. 

\begin{prop} [Lickorish \cite{LICK}] \label{prop:Lickorish_not_collapsible}
	Let $\mathfrak{L}$ be a knot 
on $m$ edges in the $1$-skeleton of a simplicial $3$-sphere $S$. Suppose that
  $S - \Delta$ is collapsible, where $\Delta$ is some tetrahedron
  in $S - \mathfrak{L}$. Then  $|S| - |\mathfrak{L}|$ is homotopy equivalent to a
  connected cell complex with one 0-cell and at most $m$ $1$-cells.
 In particular, the fundamental group of
  $|S|-|\mathfrak{L}|$ admits a presentation with $m$ generators.
\end{prop}

Now assume that a certain sphere $S$ containing a knot $\mathfrak{L}$ is LC. By Corollary
\ref{thm:corollarycollapse}, $S-\Delta$ is collapsible, for any
tetrahedron $\Delta$ not in the knot $\mathfrak{L}$. Hence by
Lickorish's criterion the fundamental group $\pi_1
\left(|S| - |\mathfrak{L}|\right)$ admits a presentation with $m$~generators. ~
 
\begin{thm}\label{thm:short}
Any $3$-sphere with a $3$-complicated $3$-edge knot is not LC. More
generally, a $3$-sphere with an $m$-gonal knot cannot be LC if the knot
is at least $m$-complicated.
\end{thm}

\begin{example} \label{thm:examplebing} As in the construction of the classical 
  ``Furch--Bing ball'' \cite[p.~73]{FUR} \cite[p.~110]{BING} \cite{ZIE}, we drill a
  hole into a finely triangulated $3$-ball along a triple pike dive of
  three consecutive trefoils; we stop drilling one step before destroying
  the property of having a ball (see Figure~\ref{fig:nonLC}).  If we add a cone over the boundary, the
  resulting sphere has a three edge knot which is a connected sum
  of three trefoil knots. By Goodrick \cite{GOO} the
  connected sum of $m$ copies of the trefoil knot is at least
  $m$-complicated. So, this sphere has a knotted triangle, the
  fundamental group of whose complement has no presentation with 
  $3$~generators. Hence $S$ cannot be LC.
\end{example}

\begin{figure}[htbf]
  \centering 
  \includegraphics[width=55mm]{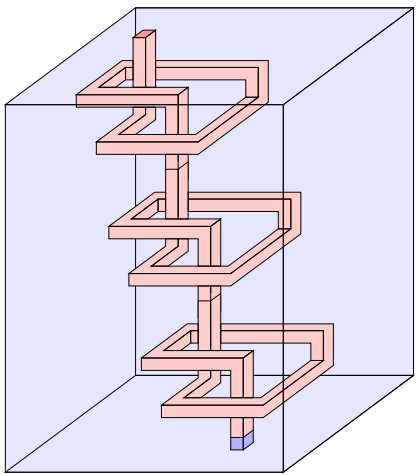} 
\caption{\small Furch--Bing ball with a (corked) tubular hole along a triple-trefoil knot. 
The cone over the boundary of this ball is a sphere that is \emph{not} LC.}
  \label{fig:nonLC}
\end{figure}

From this we get a negative answer to the Durhuus--Jonsson conjecture:

\begin{cor} \label{thm:strictcontainment2}
  Not all simplicial 3-spheres are LC.
\end{cor}

Lickorish proved also a higher-dimensional statement, basically by taking successive suspensions of the 3-sphere in Example \ref{thm:examplebing}.

\begin{thm}[Lickorish \cite{LICK}] For each $d \geq 3$, there exists a PL $d$-sphere  $S$ such that $S - \Delta$ is not collapsible for any facet $\Delta$ of $S$.
\end{thm}

To exploit our Theorem \ref{thm:Dcollapse} we need a sphere $S$ such that $S - \Delta$ is not even collapsible to a $(d-2)$-complex. To establish that such a sphere exists, we strengthen Lickorish's result.

\begin{definition} \label{thm:dual}
Let $K$ be a $d$-manifold, $A$ an $r$-simplex in $K$, and $\hat{A}$ the barycenter of $A$. Consider the barycentric subdivision $sd(K)$ of $K$. The \textit{dual} $A^*$ of $A$ is the subcomplex of $sd(K)$ given by all flags 
\[A \subset A_0 \subset A_1 \subset \cdots \subset A_r\]
where $r=\dim A$, and $\dim A_{i+1} = \dim A_i + 1$ for each $i$.
\end{definition}

$A^*$ is a cone with apex $\hat{A}$, and thus collapsible. 
If $K$ is PL (see e.g.\ Hudson~\cite{Hudson} for the definition), 
we can say more:

\begin{lemma}[{\cite[Lemma~1.19]{Hudson}}] \label{thm:newman}
Let $K$ be a PL $d$-manifold (without boundary), and let $A$ be a simplex in $K$ of dimension $r$. Then
\begin{compactitem}
\item $A^*$ is a $(d-r)$-ball, and
\item if $A$ is a face of an $(r+1)$-simplex $B$, then $B^*$ is a $(d-r-1)$-subcomplex of $\partial \, A^*$.
\end{compactitem}
\end{lemma}
 
\noindent
We have observed in Lemma \ref{thm:Scollapse} that for any $d$-sphere $S$
and any facet $\Delta$ the ball
$S - \Delta$ is collapsible onto a $(d-1)$-complex: In other words, 
via collapses one can always get \textit{one} dimension down. 
To get \textit{two} dimensions down is not so easy:
Our Theorem \ref{thm:Dcollapse} states that $S - \Delta$ is 
collapsible onto a $(d-2)$-complex precisely when $S$ is LC.

This ``number of dimensions down you can get by collapsing'' can be related to 
the minimal presentations of certain homotopy groups. The idea of the next theorem is that 
if one can get $k$ dimensions down by collapsing a manifold minus one facet, then 
the $(k-1)$-th homotopy group of the complement 
of any $(d-k)$-subcomplex of the manifold cannot be too complicated to present.

\begin{thm} \label{thm:thm0} Let $t$, $d$ with $0 \leq t \leq d-2$,
and let $K$ be a PL $d$-manifold (without boundary). Suppose that $K- \Delta$ collapses onto a $t$-complex, 
for some facet $\Delta$ of $K$. Then, for each $t$-dimensional subcomplex $\mathfrak{L}$ of $K$, 
the homotopy group 
\[\pi_{d-t-1}  (|K| - |\mathfrak{L}| )\]
has a presentation with 
$f_t ( \mathfrak{L} )$ generators,
while $\pi_i(|K| - |\mathfrak{L}| )$ is trivial for $i< d-t-1$.
\end{thm}

\begin{proof}
As usual, we assume that the collapse of $K- \Delta$ is ordered so that:
\begin{compactitem}[--]
\item first all pairs $((d-1)\textrm{-face},\; d\textrm{-face} )$ are collapsed;
\item then all pairs $((d-2)\textrm{-face},\; (d-1)\textrm{-face} )$ are collapsed;
\item $\vdots$
\item finally, all pairs $( t\textrm{-face},\; (t+1)\textrm{-face} )$ are collapsed.
\end{compactitem}
 
\noindent
Let us put together all the faces that appear above, maintaining their order, to form a single list of simplices
\[A_1, A_2, \ldots , A_{2M-1}, A_{2M}. \]
In such a list $A_1$ is a free face of $A_2$; $A_3$ is a free face of $A_4$ with respect to the complex $K - A_1 - A_2$; and so on. In general, $A_{2i-1}$ is a face of $A_{2i}$ for each $i$, and in addition, if $j > 2i$, $A_{2i-1}$ is not a face of $A_{j}$. 

We set $X_0 = A_0 := \hat{\Delta}$ and define a finite sequence $X_1, \ldots, X_M$ of subcomplexes of $sd(K)$ as follows:
\[ X_j := \bigcup 
\left\{ A_i^*  \hbox{ s.t. } i \in \{0, \ldots, 2j \} \hbox{ and } A_i \notin \mathfrak{L} \right\} 
, \qquad  \hbox{ for } j \in \{ 1, \ldots, M\}.
\]
None of the $A_{2i}$'s can be in $\mathfrak{L}$, because $\mathfrak{L}$ is $t$-dimensional and $\dim A_{2i} \geq \dim A_{2M} = t+1$. However, exactly $f_t ( \mathfrak{L} )$ of the $A_{2i-1}$'s  are in $\mathfrak{L}$. Consider how $X_j$ differs from $X_{j-1}$. There are two cases:
\begin{compactitem}[\textbullet{} ]
\item If $A_{2j-1}$ is not in $\mathfrak{L}$, 
\[ X_j = X_{j-1} \; \cup \; A^*_{2j-1} \; \cup \; A^*_{2j} ..\]
By Lemma \ref{thm:newman}, setting  $r=\dim A_{2j-1}$,   
$A^*_{2j-1}$ is a $(d-r)$-ball that contains in its boundary the $(d-r-1)$-ball $A^*_{2j}$. Thus $|X_j|$ is just $|X_{j-1}|$ with a $(d-r)$-cell attached via a cell in its boundary, and such an attachment does not change the homotopy type.
\item If $A_{2j-1}$ is in $\mathfrak{L}$, then
\[ X_j = X_{j-1} \; \cup \; A^*_{2j} .\]
As this occurs only when $\dim A_{2j-1}=t$, we have that $\dim A_{2j}=t+1$ and $\dim A^*_{2j}=d-t-1$; hence $|X_j|$ is just $|X_{j-1}|$ with a $(d-t-1)$-cell attached via its whole boundary.
\end{compactitem}
Only in the second case the homotopy type of $|X_j|$ changes at all, and this second case occurs exactly $f_t ( \mathfrak{L} )$ times. Since $X_0$ is one point, it follows that $X_M$ is homotopy equivalent to a bouquet of $f_t ( \mathfrak{L} )$ many $(d-t-1)$-spheres.

Now let us list by (weakly) decreasing dimension the faces of $K$ that do not appear in the previous list $A_1, A_2, \ldots , A_{2M-1}, A_{2M}$. We name the elements of this list
\[ A_{2M+1}, A_{2M+2}, A_F\]
(where $\sum_{i=1}^d f_i (K) = F + 1$ because all faces appear in $A_0, \ldots, A_F$). 

Correspondingly, we recursively define a new sequence of subcomplexes of $sd(K)$ setting $Y_0 := X_M$ and
\[Y_h := \left \{ \begin{array}{ll}
                Y_{h-1} \; & \hbox{ if } A_{2M + h} \in \mathfrak{L} , \\
		Y_{h-1} \; \cup \; A^*_{2M + h} & \hbox{ otherwise. }
\end{array} \right.\]
Since $\dim A_{2M+h} \leq \dim A_{2M+1} = t$, we have that $|Y_h|$ is just $|Y_{h-1}|$ with possibly a cell of dimension at least $d-t$ attached via its whole boundary. Let us consider the homotopy groups of the $Y_h$~'s : Recall that $Y_0$ was homotopy equivalent to a bouquet of $f_t ( \mathfrak{L} )$ $(d-t-1)$-spheres. Clearly, for all $h$,
\[\pi_{j}(Y_h) = 0 \hbox{ for each } j \in \{1, \ldots, d-t-2 \}.\]
Moreover, the higher-dimensional cell attached to $|Y_{h-1}|$ to get $|Y_h|$ corresponds to the addition of relators to a presentation of $\pi_{d-t-1}(Y_{h-1})$ to get a presentation of $\pi_{d-t-1}(Y_{h})$. This means that for all $h$ the group  $\pi_{d-t-1}(Y_h)$ is generated by (at most) $f_t ( \mathfrak{L} )$ elements.

The conclusion follows from the fact that, by construction, $Y_{F-2M}$ is the subcomplex of $sd(K)$ consisting of all simplices of $sd(K)$ that have no vertex in $sd(\mathfrak{L})$; and one can easily prove (see \cite[Lemma 1]{LICK}) that such a complex is a deformation retract of $|K| - |\mathfrak{L}|$.
\end{proof}

\begin{cor} \label{thm:cor1}
Let $S$ be a PL $d$-sphere with a $(d-2)$-dimensional subcomplex $\mathfrak{L}$. If the fundamental group of 
$|S| - |\mathfrak{L}|$ has no presentation with $f_{d-2} ( \mathfrak{L} )$ generators, then 
$S$ is not LC.
\end{cor}

\begin{proof}
Set $t=d-2$ in Theorem \ref{thm:thm0}, and apply Theorem \ref{thm:Dcollapse}.
\end{proof}

\begin{cor} \label{thm:cor2}
Fix an integer $d \geq 3$. Let $S$ be a $3$-sphere with an $m$-gonal knot in its 1-skeleton, so that the knot is at least $(m \cdot 2^{d-3})$-complicated. Then the $(d-3)$-rd suspension of $S$ is a PL $d$-sphere that is not LC.
\end{cor}

\begin{proof}
Let $S'$ be the $(d-3)$-rd suspension of $S$, and let $\mathfrak{L}'$ be the subcomplex of $S'$ obtained taking the $(d-3)$-rd suspension of the $m$-gonal knot $\mathfrak{L}$. Since $|S| - |\mathfrak{L}|$ is a deformation retract of $|S'| - |\mathfrak{L}'|$, they have the same homotopy groups. In particular, the fundamental group of $|S'| - |\mathfrak{L}'|$ has no presentation with $m \cdot 2^{d-3}$ generators. Now $\mathfrak{L}'$ is $(d-2)$-dimensional, and 
\[f_{d-2} (\mathfrak{L}' ) = 2^{d-3} \cdot f_1 (\mathfrak{L}) = m \cdot 2^{d-3},\]
whence we conclude via Corollary \ref{thm:cor1}, 
since all $3$-spheres are PL (and the PL property is maintained by suspensions).
\end{proof}

\begin{cor}\label{thm:cor3}
For every $d \geq 3$, not all PL $d$-spheres are LC.
\end{cor}
 
Theorem \ref{thm:thm0} can be used in connection with the existence 
of $2$-knots, that is, $2$-spheres embedded in a $4$-sphere in a knotted way 
(see Kawauchi \cite[p.~190]{KAWA}), to see that there are many non-LC $4$-spheres
beyond those that arise by suspension of $3$-spheres. Thus, being ``non-LC'' 
is not simply induced by classical knots.
 
\subsection{Many spheres are LC}
 
Next we show that all constructible manifolds are LC.

\begin{lemma} \label{lem:LCdecomposition} Let $C$ be a
  $d$-pseudomanifold. If $C$ can be split in the form $C = C_1 \cup C_2$, 
  where $C_1$ and $C_2$ are LC $d$-pseudomanifolds 
  and $C_1 \cap C_2$ is a strongly connected $(d-1)$-pseudomanifold, then $C$ is LC.
\end{lemma}

\begin{proof} 
Notice first that $C_1 \cap C_2 = \partial C_1 \cap \partial C_2$. 
In fact, every ridge of $C$ belongs to at most two facets of~$C$, 
hence every $(d-1)$-face $\sigma$ of $C_1 \cap C_2$
is contained in exactly one $d$-face of $C_1$ and in exactly one 
$d$-face of $C_2$.  
 
Each $C_i$ is LC; let us fix a local construction for each of them, and call $T_i$ 
the tree along which $C_i$ is locally constructed. Choose some $(d-1)$-face $\sigma$
in $C_1\cap C_2$, which thus specifies a $(d-1)$-face in the boundary of $C_1$ and of $C_2$.
Let $C'$ be the 
pseudomanifold obtained attaching $C_1$ to $C_2$ along the 
two copies of $\sigma$. $C'$ can be locally constructed along 
the tree obtained by joining $T_1$ and $T_2$ by an edge across $\sigma$: Just redo the same moves of 
the local constructions of the $C_i$'s. So $C'$ is LC. 

If $C_1 \cap C_2$ consists of one simplex only, then $C' \equiv C$ 
and we are already done. Otherwise, by the strongly connectedness assumption, 
the facets of $ C_1 \cap  C_2$ can be labeled $0, 1, \ldots, m$, so that:
\begin{compactitem}
\item the facet labeled by $0$ is $\sigma$;
\item each facet labeled by $k \geq 1$ is adjacent to some facet labeled $j$ with $j < k$.
\end{compactitem}
Now for each $i \geq 1$, glue together the two copies of the facet $i$ inside $C'$. 
All these gluings are \emph{local} because of the labeling chosen, and we eventually obtain $C$. Thus, $C$ is LC.
\end{proof}

Since all constructible simplicial complexes are pure and strongly
connected \cite{BJOE}, we obtain for simplicial $d$-pseudomanifolds that
\[ \{ \hbox{constructible} \} \ \subseteq\ \{ \hbox{LC} \}. \]

The previous containment is strict: Let $C_1$ and $C_2$
be two LC simplicial $3$-balls on $7$ vertices consisting of $7$ tetrahedra, as indicated in Figure~\ref{fig:nonLCmanifold}. 
(The $3$-balls are cones over the subdivided triangles on their fronts.)

\begin{figure}[htbf]
  \centering 
  \includegraphics[width=0.58\linewidth]{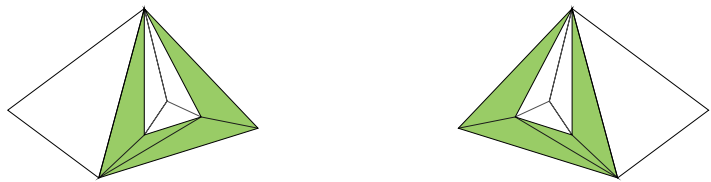} 
\caption{\small Gluing the simplicial $3$-balls along the shaded $2$-dimensional subcomplex
gives an LC, non-constructible $3$-pseudomanifold.}
  \label{fig:nonLCmanifold}
\end{figure}

Glue them together in the shaded strongly connected subcomplex in their boundary (which uses
$5$ vertices and $4$ triangles).  The resulting simplicial complex $C$,
on $9$ vertices and $14$ tetrahedra, is LC by Lemma \ref{lem:LCdecomposition},
but the link of the top vertex is an annulus, and hence not LC.
In fact, the complex $C$ is not constructible, since the link of the top vertex is not
constructible. Also, $C$ is not $2$-connected, it retracts to a $2$-sphere.
So, LC $d$-pseudomanifolds are not necessarily $(d-1)$-connected.
Since all constructible $d$-complexes are $(d-1)$-connected, and
every constructible $d$-pseudomanifold is either a $d$-sphere or a $d$-ball \cite[Prop.~1.4, p.~374]{HS}, 
the previous argument produces many examples of $d$-pseudomanifolds with boundary that are LC, 
but not constructible. 

None of these examples, however, will be a sphere (or a ball).
We will prove in Theorem~\ref{thm:LCnonconstructibleballs} 
that there are LC $3$-balls that are not constructible; we show now that for $d$-spheres, for every $d \geq 3$, the containment 
$\{ \textrm{constructible} \} \subseteq \{ \textrm{LC} \}$ is strict.

\begin{lemma} \label{thm:suspensions}
Suppose that  a $3$-sphere $\bar{S}$ is LC but not constructible. 
Then for all $d \geq 3$, the $(d-3)$-rd suspension of $\bar{S}$ is a $d$-sphere that is also LC but not constructible.
\end{lemma}

\begin{proof}
Whenever $S$ is an LC sphere, $v*S$ is an LC $(d+1)$-ball. 
(The proof is straightforward from the definition of ``local construction''.) 
Thus the suspension $(v*S) \cup (w*S)$ is also LC by 
Lemma~\ref{lem:LCdecomposition}. 
On the other hand, the suspension of a non-constructible sphere is a non-constructible sphere 
\cite[Corollary 2]{HZ}.
\end{proof}
 
Of course, we should better show that the $3$-sphere $\bar{S}$ in the assumption of Lemma \ref{thm:suspensions} really exists.
This will be established in Example \ref{thm:examplelick}, using Corollary \ref{thm:corollarycollapse} as follows.

\begin{lemma} \label{thm:coneoverbeethoven} Let $B$ be a $3$-ball, $v$
  an external point, and $B \cup v * \partial B$ the $3$-sphere
  obtained by adding to $B$ a cone over its boundary. If $B$ is
  collapsible, then $B \cup v * \partial B$  is LC.
\end{lemma}

\begin{proof}
  By Corollary \ref{thm:corollarycollapse}, and since $B$ is collapsible, all we need to prove is
  that $(B \cup v * \partial B)  - (v*\sigma)$ collapses onto $B$, for some triangle $\sigma$ in
  the boundary of $B$. 

As all $2$-balls are collapsible, and $\partial B - \sigma $ is a $2$-ball, there 
is some vertex $P$ in $\partial B$ such that $\partial B - \sigma \searrow P$. 
This naturally induces a collapse of $v*\partial B \; - \; v*\sigma$ onto $\partial B \, \cup \, v*P$,
according to the correspondence
\[\sigma \hbox{ is a free face of } \Sigma \; \; \Longleftrightarrow \; \; v*\sigma \hbox{ is a free face of } v*\Sigma.\] 

Collapsing the edge $v*P$ down to $P$,  
we get $v*\partial B \; - \; v*\sigma \searrow \partial B$.  

In the collapse given here, the pairs of faces removed are all of the form $(v*\sigma, v*\Sigma)$; 
thus, the $(d-1)$-faces in $\partial B$ are removed together with subfaces (and not with superfaces) in the collapse. This means that the freeness of the faces in $\partial B$ is not needed; so when we glue back $B$ the collapse
$v*\partial B \; - \; v*\sigma \searrow \partial B$
can be read off as
$B \;\cup \; v * \partial B \; - \; v*\sigma \; \searrow \; B.$
\end{proof}
 
\begin{example} \label{thm:examplelick} In \cite{LICKMAR}, Lickorish
  and Martin described a collapsible $3$-ball $B$ with a knotted spanning edge. 
  This was also obtained
  independently by Hamstrom and Jerrard \cite{HAM}. The knot is an
  arbitrary $2$-bridge index knot (for example, the trefoil knot). 
  Merging $B$ with the cone over its
  boundary, we obtain a knotted 3-sphere $\bar{S}$ which is LC (by Lemma
  \ref{thm:coneoverbeethoven}; see also \cite{LICK}) but not   
  constructible (because it is knotted; see \cite[p.~54]{Hthesis} or \cite{HZ}).
\end{example}

\begin{remark} \label{thm:remarklick}
  In his 1991 paper \cite[p.~530]{LICK}, Lickorish announced (for a proof see \cite[pp.~100--103]{Benedetti-diss}) that ``with a little ingenuity'' one can get a sphere $S$ with a $2$-complicated triangular knot
  (the double trefoil), such that $S - \Delta$ is collapsible. Such a
  sphere is LC by Corollary \ref{thm:corollarycollapse}.
\end{remark}

\begin{example} \label{ex:knottedLC}
The triangulated knotted $3$-sphere $S^3_{13, 56}$ realized by Lutz \cite{LUTZ1} 
	has $13$ vertices and $56$ facets. 
Since it contains a $3$-edge trefoil knot in its $1$-skeleton, $S^3_{13, 56}$ cannot be constructible, 
according to Hachimori and Ziegler~\cite{HZ}. 

Let $B_{13, 55}$ be the $3$-ball obtained removing the facet $\Delta = \{1,2,6,9\}$ from $S^3_{13, 56}$. Let $\sigma$ be the triangle $\{2,6,9\}$. Then $B^3_{13, 55}$ collapses to the $2$-disc $\partial \Delta - \sigma$ (F. H. Lutz, personal communication; see \cite[pp.~106--107]{Benedetti-diss}). All $2$-discs are collapsible. In particular, $B^3_{13, 55}$ is collapsible, so $S^3_{13, 56}$ is LC. 
\end{example}

\begin{cor}
For each $d \geq 3$, not all LC $d$-spheres are constructible. In particular, a knotted
$3$-sphere can be LC (but it is not constructible) if the knot is just $1$-complicated or $2$-complicated. 
\end{cor}

The knot in the $1$-skeleton of the ball $B$ in Example
\ref{thm:examplelick} consists of a path on the boundary of $B$
together with a ``spanning edge'', that is, an edge in the interior of
$B$ with both extremes on $\partial B$. This edge determines the knot,
in the sense that any other path on $\partial B$ between the two
extremes of this edge closes it up into an equivalent knot. For these
reasons such an edge is called a \emph{knotted spanning edge}. More
generally, a \emph{knotted spanning arc} is a path of edges in the  
interior of a $3$-ball, such that both extremes of the path lie on the
boundary of the ball, and any boundary path between these extremes closes it into a knot.
(According to this definition, the
  relative interior of a knotted spanning arc is allowed to intersect the boundary of
  the $3$-ball; this is the approach of Hachimori
  and Ehrenborg in \cite{EH}.)

The Example \ref{thm:examplelick} can then be generalized by adopting
the idea that Hamstrom and Jerrard used to prove their ``Theorem B''
\cite[p.~331]{HAM}, as follows.

\begin{thm} \label{thm:HJ} Let $K$ be any $2$-bridge knot (e.g. the
 trefoil knot). For any positive integer $m$, there exists a
 collapsible $3$-ball $B_m$ with a knotted spanning arc of $m$ edges,
 such that the knot is the connected union of $m$ copies of $K$.
\end{thm}

\begin{proof} 
By the work of Lickorish--Martin \cite{LICKMAR} 
(see also \cite{HAM} and Example \ref{thm:examplelick}) there exists a collapsible $3$-ball 
$B$ with a knotted spanning edge $[x, y]$, the knot being $K$. So if $m=1$ 
we are already done.

Otherwise, take $m$ copies $B^{(1)}$, $\ldots$, $B^{(m)}$ of the ball $B$ and glue them all together by identifying the vertex $y^{(i)}$ 
of $B^{(i)}$ with the vertex $x^{(i+1)}$ of $B^{(i+1)}$, for each 
$i$ in $\{1, \ldots, m-1\}$. The result is a cactus of $3$-balls $C_m$. By induction on $m$, it is easy to see that a cactus of $m$~collapsible $3$-balls is collapsible. To obtain a $3$-ball from $C_m$, we thicken the junctions between the $3$-balls by attaching $m-1$ square pyramids with apex $y^{(i)} \equiv x^{(i+1)}$. Each pyramid can be triangulated into two tetrahedra to make the final complex simplicial. Let $B_m$ be the resulting $3$-ball. All the spanning edges of the $B^{(i)}$'s are concatenated in $B_m$ to yield a knotted spanning arc of $m$~edges, the knot being equivalent to the $m$-ple 
connected union of $K$ with himself. Moreover, the ``extra pyramids'' introduced can be collapsed away. This yields a collapse of the ball $B_m$ onto the complex $C_m$, which is collapsible.
\end{proof}

\begin{cor} \label{thm:LCknotted} A $3$-sphere with an $m$-complicated
  $(m+2)$-gonal knot can be LC.
\end{cor}

\begin{proof}
  Let $S_m = B_m \cup (v*\partial B_m)$, where $B_m$ is the $3$-ball constructed
  in the previous theorem. By Lemma \ref{thm:coneoverbeethoven},
  $S_m$ is LC. The spanning arc of $m$ edges is closed up in $v$
  to form an $(m+2)$-gon.
\end{proof}

\begin{remark} \label{rem:LCknotted}
The bound given by Corollary \ref{thm:LCknotted} can be improved: In fact, for each positive integer $m$ there exists an LC $3$-sphere with an $(m+1)$-complicated $(m+2)$-gonal knot. The proof is rather long, so we preferred to omit it, referring the reader to \cite[pp.~100--103]{Benedetti-diss}.
\end{remark}

The spheres mentioned in Corollary~\ref{thm:LCknotted} and Corollary Remark~\ref{rem:LCknotted} are not vertex decomposable, not shellable and not constructible, because of the following 
result about the bridge index.

\begin{thm} [Ehrenborg, Hachimori, Shimokawa \cite{EH} \cite{HS}] \label{thm:EHS} 
Suppose that a $3$-sphere (or a $3$-ball) $S$ contains a knot of $m$ edges.
\begin{compactitem} 
\item[--] If the bridge index of the knot exceeds $\frac{m}{3}$, then
  $S$ is not  vertex decomposable;
\item[--] If the bridge index of the knot exceeds $\frac{m}{2}$, then
  $S$ is not constructible.
\end{compactitem}
\end{thm}

\noindent The bridge index of a $t$-complicated knot is at least $t+1$. So, if a knot is at least $\lfloor\frac{m}{3}\rfloor$-complica\-ted, 
its bridge index automatically exceeds $\frac{m}{3}$. 
Thus, Ehrenborg--Hachimori--Shimokawa's theorem, the results of
Hachimori and Ziegler in \cite{HZ}, the previous examples, and our
present results blend into the following new hierarchy.

\begin{thm} A $3$-sphere with a non-trivial knot consisting of\\ 
\begin{tabular}{rl}
$3$ edges, $1$-complicated  & is not constructible, but can be LC. \\
$3$ edges, $2$-complicated  & is not constructible, but can be LC. \\
\phantom{xxxxxxxx}
$3$ edges, $3$-complicated or more  & is not LC. \\
$4$ edges, $1$-complicated & is not vertex dec., but can be shellable.  \\
$4$ edges, $2$- or $3$-complicated & is not constructible, but can be LC.  \\
$4$ edges, $4$-complicated or more & is not LC.  \\
$5$ edges, $1$-complicated & is not vertex dec., but can be shellable.  \\
$5$ edges, $2$-, $3$- or $4$-complicated  & is not constructible, but can be LC.  \\
$5$ edges, $5$-complicated or more & is not LC.  \\
$6$ edges, $1$-complicated & can be vertex decomposable.  \\
$6$ edges, $2$-complicated & is not vertex dec., but can be LC.  \\
$6$ edges, $3$-, $4$ or $5$-complicated & is not constructible, but can be LC.  \\
$6$ edges, $6$-complicated or more  & is not LC. \\
 \vdots\qquad{}  & \qquad \vdots\\
$m$ edges, $k$-complicated, $k \geq \lfloor\frac{m}{3}\rfloor$  & is not vertex decomposable. \\
$m$ edges, $k$-complicated, $k \geq \lfloor\frac{m}{2}\rfloor$  & is not constructible. \\
$m$ edges, $k$-complicated, $k \leq m-1$  & can be LC. \\
$m$ edges, $k$-complicated, $k \geq m$  & is not LC. 
\end{tabular} 
\end{thm}

\noindent
The same conclusions are valid for $3$-balls that contain a knot, up to replacing the word ``LC'', wherever it occurs, with the word ``collapsible''. (See Lemma~\ref{thm:coneoverbeethoven}, Corollary~\ref{thm:mainballs} and \cite{HZ}.)

\smallskip
One may also derive from  Zeeman's theorem (``given any
simplicial $3$-ball, there is a positive integer $r$ so that its
$r$-th barycentric subdivision is collapsible'' \cite[Chapters I and III]{Zeeman})
that any $3$-sphere will become 
LC after sufficiently many barycentric subdivisions.
On the other hand, there is no fixed number $r$ of subdivisions that is sufficient
to make \emph{all} $3$-spheres LC. (For this use sufficiently complicated knots,
together with Theorem~\ref{thm:short}.)
 
\section{On LC Balls}

The combinatorial topology of $d$-balls and of $d$-spheres are intimately related:
Removing any facet $\Delta$ from a 
 $d$-sphere $S$ we obtain a $d$-ball $S-\Delta$, and 
adding a cone over the boundary of a $d$-ball $B$ we obtain a $d$-sphere $S_B$.
We do have a combinatorial characterization of LC $d$-balls,
which we will reach in Theorem \ref{thm:mainDballs}; it is a bit more complicated, but otherwise analogous to the
characterization of LC $d$-spheres as given in Main Theorem~\ref{mainthm:hierarchyspheres}.
 
\begin{thm} \label{thm:hierarchyballs}
For simplicial $d$-balls, we have the following hierarchy: \em 
\[ \Big\{\begin{array}{@{}c@{}}
	\textrm{vertex}\\ \textrm{decomp.}
    \end{array}\Big\} \subsetneq 
   \{\textrm{shellable}\} \subsetneq
   \{\textrm{constructible}\} \subsetneq  
   \{\textrm{LC}\} \subsetneq
   \Big\{\begin{array}{@{}c@{}}
   \textrm{collapsible onto a}\\
   (d-2)\textrm{-complex}
   \end{array}\Big\} \subsetneq
   \{\textrm{all $d$-balls}\}. 
\]
\end{thm}

\begin{proof}
	The first two inclusions are known.
	We have already seen that all constructible complexes are LC (Lemma~\ref{lem:LCdecomposition}).
	Every LC $d$-ball is collapsible onto a $(d-2)$-complex 
	by Corollary \ref{cor:LCcollapsible-dimd}.  
	
	Let us see next that all inclusions are strict for $d=3$:
	For the first inclusion this follows from Lockeberg's example of a $4$-polytope
	whose boundary is not vertex decomposable.
	For the second inclusion, take Ziegler's non-shellable ball from \cite{ZIE},
	which is constructible by construction.
	A non-constructible $3$-ball that is LC will be provided by Theorem \ref{thm:LCnonconstructibleballs}.
	A collapsible $3$-ball that is not LC will be given in Theorem~\ref{thm:collapsiblenonLC}.
	Finally, Bing and Goodrick showed that not every $3$-ball is collapsible \cite{BING, GOO}.
	
	To show that the inclusions are strict for all $d\ge3$, we argue as follows.
	For the first four inclusions we get this from the case $d=3$, since 
	\begin{compactitem}[ -- ]
	\item cones are always collapsible,
	\item the cone $v*B$ is vertex decomposable resp.\ shellable resp.\ constructible if and only if $B$ is, 
	\item and in Proposition \ref{prop:LCcones} we will show that $v * B$ is LC if and only if $B$ is.
	\end{compactitem}
	For the last inclusion and $d\ge3$, we look at the 
	$d$-balls obtained by removing a facet from a non-LC $d$-sphere.
	These exist by Corollary~\ref{thm:cor2}; they do not collapse onto a $(d-2)$-complex by 
	Theorem~\ref{thm:Dcollapse}.  
\end{proof}

\subsection{Local constructions for $d$-balls}

We begin with a relative version of the notions of ``facet-killing sequence'' and ``facet massacre'', which we introduced in Subsection~\ref{sec:notLCspheres}. 

\begin{definition} Let $P$ a pure $d$-complex. Let $Q$ be a proper subcomplex of $P$, either pure $d$-dimensional or empty. 
A \textit{facet-killing sequence of $(P ,Q)$} is a sequence $P_0, P_1, \ldots, P_{t-1}, P_t$ of  simplicial complexes such that $t=f_d(P)-f_d(Q)$, $P_0=P$, and $P_{i+1}$ is obtained by $P_i$ removing a pair $(\sigma, \Sigma)$ such that $\sigma$ is a free $(d-1)$-face of $\Sigma$ that does not lie in~$Q$ 
(which also implies that $\Sigma\notin Q$).
 \end{definition}

It is easy to see that $P_t$ has the same 
$d$-faces as $Q$. The version of facet killing sequences given in Definition \ref{thm:facetkilling} is a special case of this one, namely the case when $Q$ is empty.

\begin{definition}
Let $P$ a pure $d$-dimensional simplicial complex. Let $Q$ be either the empty complex, or a pure $d$-dimensional proper subcomplex of $P$. A \textit{pure facet-massacre of $(P,Q)$} is a sequence $P_0, P_1, \ldots, P_{t-1}, P_t$ of (pure) complexes such that $t=f_d(P) - f_d(Q)$, $P_0=P$, and $P_{i+1}$ is obtained by $P_i$ removing:
	\begin{compactenum}[(a)]
	\item a pair $(\sigma, \Sigma)$ such that $\sigma$ is a free $(d-1)$-face of $\Sigma$ that does not lie in $Q$, and
	\item all inclusion-maximal faces of dimension smaller than $d$ that are left after the removal of type (a) or, recursively, after removals of type (b).	
\end{compactenum}
\end{definition}

\noindent
Necessarily $P_t=Q$ (and when $Q = \emptyset$ we recover the notion of facet-massacre of $P$, that we introduced in Definition \ref{thm:facetmassacre}). 
It is easy to see that a step $P_i \longrightarrow P_{i+1}$ can be factorized (not in an unique way) in an elementary collapse followed by a removal of faces of dimensions smaller than $d$ that makes $P_{i+1}$ a pure complex. Thus, a single pure facet-massacre of $(P,Q)$ corresponds to many facet-killing sequences of $(P,Q)$. 

We will apply both definitions to the pair $(P,Q) = (K^T, \partial B)$, where $K^T$ is defined for balls as follows.

\begin{definition}
If $B$ be a $d$-ball with $N$ facets, and $T$ is a spanning tree of
  the dual graph of $B$, define $K^T$ as the subcomplex of $B$ formed by
all $(d-1)$-faces of $B$ that are not hit by $T$.
\end{definition}

\begin{lemma} \label{thm:Bcollapse} Under the previous notations, 
\begin{compactitem}
\item $K^T$ is a pure $(d-1)$-dimensional simplicial complex,
  containing $\partial B$ as a subcomplex;
\item $K^T$ has $D + \frac{b}{2}$ facets, where $b$ is the number of
  facets in $\partial B$, and $D:=\frac{dN-N+2}{2}$;
\item for any $d$-simplex $\Delta$ of $B$, $\; \;B - \Delta \; \searrow K^T$;
\item $K^T$ is homotopy equivalent to a $(d-1)$-dimensional sphere.
\end{compactitem}
\end{lemma}

We introduce another convenient piece of terminology.

\begin{definition}[seepage]  Let $B$ be a simplicial $d$-ball. 
	A \emph{seepage} is a $(d-1)$-dimensional subcomplex $C$ of $B$ 
	whose $(d-1)$-faces are exactly given by the boundary of $B$.
\end{definition}

A seepage is not necessarily pure; actually there is only one pure seepage, 
namely $\partial B$ itself. Since $K^T$ contains $\partial B$, a collapse of $K^T$ onto
 a seepage must remove all the $(d-1)$-faces of $K^T$ that are not in $\partial B$:
This is what we called a facet-killing sequence of $(K^T, \partial B)$.

\begin{proposition} \label{thm:Balongtrees} Let $B$ be a $d$-ball, and
  $\Delta$ a $d$-simplex of $B$. Let $C$ be a seepage of $\partial B$. Then,
\[ B- \Delta \; \searrow \; C \; \; \Longleftrightarrow \; \exists \; T \hbox{ s.t. } \; K^T \; \searrow C. \]
\end{proposition}

\begin{proof}
Analogous to the proof of Proposition \ref{thm:Salongtrees}. 
The crucial assumption is that no face of $\partial B$ 
is removed in the collapse (since all boundary faces are still
  present in the final complex $C$).
\end{proof}

If we fix a spanning tree $T$ of the dual graph of $B$,  
we have then a 1-1 correspondence between the following sets: 
\begin{compactenum}
\item the set of collapses $B- \Delta \; \searrow K^T$;
\item the set of ``natural labelings'' of $T$, where $\Delta$ is labeled by $1$;
\item the set of the first parts $(T_1, \ldots, T_N)$ of local constructions for $B$, with $T_1 = \Delta$.
\end{compactenum}

\begin{thm} \label{thm:Bmassacre}
 Let $B$ be a $d$-ball; fix a facet $\Delta$, and a spanning tree $T$ of the dual graph of $B$, rooted at $\Delta$. The second part of a local construction for $B$ along $T$ corresponds bijectively to a facet-massacre of $(K^T , \partial B)$.
\end{thm}

\begin{proof}
Let us start with a local construction $\left[ T_1, \ldots, T_{N-1}, \right] T_N, \ldots, T_k$ for $B$ along $T$. Topologically, $B=T_N / {\sim}$, where ${\sim}$ is the equivalence relation determined by the gluing, and $K^T = \partial T_N / {\sim}$. 

$K^T$ has $D + \frac{b}{2}$ facets (see Lemma \ref{thm:Bcollapse}), and all of them, 
except the $b$ facets in the boundary, represent gluings. Thus we have to describe a 
sequence $P_0, \ldots, P_t$ with $t=D - \frac{b}{2}$. But the local construction 
$\left( T_1, \ldots, T_{N-1}, \right) T_N, \ldots, T_k$ produces $B$ (which has 
$b$ facets in the boundary) from $T_N$ (which has $2D$ facets in the boundary, 
cf.\ Lemma \ref{thm:treesA}) in $k-N$ steps, each removing a pair of facets from the boundary. 
So, $2D - 2(k-N)=b$, which implies $k-N=t$.  

Define $P_0 := K_T = \partial T_N / {\sim}$, and $P_j := \partial T_{N+j} / {\sim}$. In the first LC step, $T_N \rightarrow T_{N+1}$, we remove from the boundary a free ridge $r$, together with the unique pair $\sigma', \sigma''$ of facets of $\partial T_N$ sharing $r$. At the same time, $r$ and the newly formed face $\sigma$ are sunk into the interior; so obviously neither $\sigma$ nor $r$ will appear in $\partial B$. This step $\partial T_N \longrightarrow \partial T_{N+1}$ naturally induces an analogous step $\partial T_{N+j}/ {\sim} \longrightarrow \partial T_{N+j+1}/ {\sim}$, namely, the removal of $r$ and of the unique $(d-1)-$face $\sigma$ containing it, with $r$ not in $\partial B$. 

The rest is analogous to the proof of Theorem \ref{thm:Smassacre}. \endproof
\end{proof}

\vspace{0.01\linewidth}
\par \noindent Thus, $B$ can be locally constructed along a tree $T$ if and only if $K^T$ collapses onto some seepage. What if we do not fix the tree $T$ or the facet $\Delta$?

\begin{lemma} \label{thm:astussia} Let $B$ be a $d$-ball; let $\sigma$
  be a $(d-1)$-face in the boundary $\partial B$, and let $\Sigma$ be the unique
  facet of $B$ containing $\sigma$. Let $C$ be a subcomplex of $B$. If $C$
  contains $\partial B$, the following are equivalent:
\begin{compactenum}[ \rm1. ]
\item $B - \Sigma \; \searrow \; C;$ 
\item $B - \Sigma - \sigma \; \searrow \; C - \sigma$;
\item $B \; \searrow \; C - \sigma$.
\end{compactenum}
\end{lemma}

\begin{thm} \label{thm:mainDballs} Let $B$ be a $d$-ball. Then the following are equivalent: 
\begin{compactenum}[ \rm1. ]
\item $B$ is LC;
\item $K^T$ collapses onto some seepage $C$, for some
  spanning tree $T$ of the dual graph of $B$;
\item there exists a seepage $C$ such that for every facet $\Delta$ of $B$ one has $B - \Delta \; \searrow C$;
\item $B - \Delta \; \searrow C$, for some facet $\Delta$ of $B$, and
  for some seepage $C$;
\item there exists a seepage $C$ such that for every
 facet $\sigma$ of $\partial B$ one has $B \; \searrow \; C - \sigma$;
\item $B \; \searrow \; C - \sigma$, for some facet $\sigma$ of
  $\partial B$, and for some seepage $C$;
\end{compactenum}
\end{thm}

\begin{proof}
  The equivalences $1 \Leftrightarrow 2 \Leftrightarrow 3 \Leftrightarrow 4$ 
are established analogously to the proof of Theorem \ref{thm:Dcollapse}. 
Finally, Lemma \ref{thm:astussia} implies that $3 \Rightarrow 5 \Rightarrow 6 \Rightarrow 4$.
\end{proof}
 
\begin{cor} \label{cor:LCcollapsible-dimd}
	Every LC $d$-ball collapses onto a $(d-2)$-complex.
\end{cor}

\begin{proof}
  By Theorem~\ref{thm:mainDballs}, the ball $B$ collapses onto 
  the union of the boundary of $B$ minus a facet with some $(d-2)$-complex.
  The boundary of $B$ minus a facet is a $(d-1)$-ball; thus it can
  be collapsed down to dimension $d-2$, and the additional $(d-2)$-complex will not
  interfere. 	
\end{proof}

\begin{cor} \label{thm:mainballs} Let $B$ be a $3$-ball. Then the following are equivalent: 
\begin{compactenum}[ \rm1. ]
\item $B$ is LC;
\item $K^T \searrow \partial B$, for some spanning tree
  $T$ of the dual graph of $B$;
\item $B - \Delta \; \searrow \partial B$, for every facet $\Delta$ of $B$;
\item $B - \Delta \; \searrow \partial B$, for some facet $\Delta$ of $B$;
\item $B \; \searrow \; \partial B - \sigma$, for every facet $\sigma$ of $\partial B$;
\item $B \; \searrow \; \partial B - \sigma$, for some facet $\sigma$ of $\partial B$.
\end{compactenum}
\end{cor}

\begin{figure}[htbf]
  \centering
  \includegraphics[width=.18\linewidth]{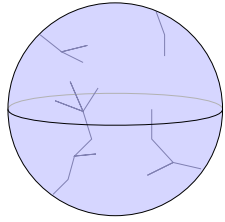}
  \caption{\small A seepage of a $3$-ball.}
  \label{fig:seepage}
\end{figure}

\begin{proof}
  When $B$ has dimension 3, any seepage $C$ of $\partial B$ is a
  $2$-complex containing $\partial B$, plus some edges and vertices. 
If a complex homotopy equivalent to $S^2$ collapses onto $C$, 
then $C$ is also homotopy equivalent to $S^2$, thus $C$ can only 
be $\partial B$ with some trees attached (see Figure \ref{fig:seepage}),
which implies that $C \searrow \partial B$. 
\end{proof}

\begin{cor}  \label{thm:LCcollapsible} All LC $3$-balls are collapsible. 
\end{cor}

\begin{proof} If $B$ is LC, it collapses to some $2$-ball $\partial B -
  \sigma$, but all $2$-balls are collapsible.
\end{proof}

\begin{cor} All constructible $3$-balls are collapsible. 
\end{cor}

For example,
  Ziegler's ball, Gr\"unbaum's ball, and Rudin's ball are collapsible (see \cite{ZIE}).

\begin{remark} The locally constructible $3$-balls with $N$ facets are precisely 
	the $3$-balls that admit a ``special collapse'', namely such that after the
    first elementary collapse, in the next $N-1$ collapses, 
    no triangle 
	of $\partial B$ is collapsed away. Such a collapse acts along a dual (directed) 
	tree of the ball, whereas a generic collapse acts along an acyclic graph that 
	might be disconnected.
\end{remark}
 
  One could argue that maybe ``special collapses'' are not that special:
  Perhaps every collapsible $3$-ball has a collapse that
  removes only one boundary triangle in its top-dimensional phase?  
 This is not so: We will produce a counterexample in the next subsection (Theorem \ref{thm:collapsiblenonLC}).

\begin{thm}\label{thm:LCnonconstructibleballs}
	For every $d\ge3$, not all LC $d$-balls are constructible.
\end{thm}

\begin{proof} 
If $B$ is a non-constructible $d$-ball and $v$ is a new vertex, then $v*B$ is a non-constructible $(d+1)$-ball. Also, it is easy to see that if $B$ is LC then $v \ast B$ is also LC (cf. Proposition \ref{prop:LCcones}). Therefore, it suffices to prove the claim for $d=3$. 

In Example~\ref{ex:knottedLC} we described a $3$-ball $B_{13,55}$ that collapses onto its boundary minus a facet. By Corollary~\ref{thm:mainballs}, $B_{13,55}$ is LC. At the same time, $B_{13,55}$ contains a $3$-edge trefoil knot, which prevents $B_{13,55}$ from being constructible \cite[Thm.~1]{HZ}.  
\end{proof}

\subsection{3-Balls without interior vertices.}
Here we show that a simplicial $3$-ball with all vertices on the boundary cannot contain any knotted spanning edge if it is LC, but might contain some if it is collapsible. 
We use this fact to establish our hierarchy for $d$-balls (Theorem \ref{thm:hierarchyballs}).

Let us fix some notation first. Recall that by Theorem \ref{thm:topologyLC3-dim}, 
each connected component of the boundary of a simplicial LC $3$-pseudomanifold 
is homeomorphic to a simply-connected union of $2$-spheres, any two of which share 
at most one point. Let us call \emph{pinch points} the points shared by two or more 
spheres in the boundary of an LC $3$-pseudomanifold. 

\begin{definition}
  \label{defn:cases}[Steps of types (i)-(ix) in LC constructions]
  Any admissible step in a local construction of a $3$-pseudomanifold falls into one of the following nine types: 
\begin{compactenum}[(i)\ ]
\item attaching a tetrahedron along a triangle;
\item identifying two boundary triangles that share exactly 1 edge;
\item identifying two boundary triangles that share 1 edge and the opposite vertex;
\item identifying two b. t. that share 2 edges that meet in a pinch point;
\item identifying two b. t. that share 2 edges that do not meet in a pinch point;
\item identifying two b. t.  that share 3 edges, all of whose vertices are pinch points;
\item identifying two b. t.  that share 3 edges, two of whose vertices are pinch points;
\item identifying two b. t.  that share 3 edges, one of whose vertices is a pinch point;
\item identifying two b. t. that share 3 edges, none of whose vertices is a pinch point.
\end{compactenum}
\end{definition}

{\noindent}For example, the first $N - 1$ steps of any local construction of a
$3$-pseudomanifold with $N$ tetrahedra are all of type (i); the last step in the
local construction of a $3$-sphere is necessarily of type~(ix).

The following table summarizes the distinguished effects of the steps:
\begin{center} \small
\begin{tabular}{rcc}
\emph{ step type } & \emph{ no.\ of interior vertices } & \emph{ no.\ of connected components of the boundary}\\ 
(i) & + 0 & + 0 \\ 
(ii) & + 0 & + 0 \\
(iii) & + 0 & $\quad \; \;$+ 0 (*)\\
(iv) & + 0 & + 1 \\
(v) & + 1 & + 0 \\
(vi) & + 0 & + 3 \\
(vii) & + 1 & + 2 \\
(viii) & + 2 & + 0 \\
(ix) & + 3  & -- 1
\end{tabular}
\end{center}
where the asterisk recalls that a type (iii) step
\emph{almost} disconnects the boundary, pinching it in a point.

Now, let $B$ be an LC $3$-ball \emph{without} interior vertices. Steps of type (v), (vii), (viii) or (ix) sink respectively one, one, two and three vertices into the interior, so they cannot occur in the local construction of $B$. Furthermore, any
  identification of type (vi) or (iv) increases the number of
  connected components in the boundary,  hence it must be followed by at least
  one step of type (ix), which destroys a connected component of the boundary.
Yet (ix) is forbidden, so no identification of type (vi) or (iv) can
  occur. Finally, the ``pinching step'' (iii) needs to be followed by one of
  the steps (vi), (vii), (viii) or (ix) in order to restore the ball topology
  -- but such steps are forbidden. This leads us to the following Lemma:

\newpage
\begin{lemma} \label{lem:sudoku}
Let $B$ be an LC $3$-pseudomanifold. The following are equivalent:
\begin{compactenum}[\rm (1)]
 \item in some local construction for $B$ all steps are of type (i) or (ii); 
\item in every local construction for $B$ all steps are of type (i) or (ii); 
\item $B$ is a $3$-ball without interior vertices.
\end{compactenum}
\end{lemma}

We will use Lemma \ref{lem:sudoku} to obtain examples of non-LC $3$-balls. We already know that non-collapsible balls are not LC, by Corollary \ref{thm:LCcollapsible}: so a $3$-ball with a knotted spanning edge cannot be LC if the knot is the sum of two or more trefoil knots. (See also Bing \cite{BING} and Goodrick \cite{GOO}.) What about balls with a spanning edge realizing a single trefoil knot?

\begin{proposition} \label{prop:knottednotLC}
An LC $3$-ball without interior vertices does not contain any knotted spanning edge.
\end{proposition}

\begin{proof}
An LC $3$-ball $B$ without interior vertices is obtained from a tree of tetrahedra via local gluings of type (ii), by Lemma \ref{lem:sudoku}. A tree of tetrahedra has no interior edge. Each type~(ii) step preserves the existing spanning edges (because it does not sink vertices into the interior), and creates one more spanning edge $e$, clearly unknotted (because the other two edges of the sunk triangle form a boundary path that ``closes up'' the edge $e$ onto an $S^1$ bounding a disc inside $B$).  
It is easy to verify that the subsequent type (ii) steps leave such edge $e$ spanning and unknotted.
\end{proof}

\begin{remark} The presence of knots/knotted spanning edges is not the only obstruction to local constructibility. Bing's thickened house with two rooms \cite{BING, HACHIweb} is a $3$-ball $B$ with all vertices on the boundary, so that every interior triangle of $B$ has at most one edge on the boundary $\partial B$. Were $B$  
LC, every step in its local construction would be of type~(ii) (by Lemma \ref{lem:sudoku}); in particular, the last triangle to be sunk into the interior of $B$ would have exactly two edges on the boundary of $B$. Thus Bing's thickened house with two rooms cannot be LC, even if it does not contain a knotted spanning edge. 
\end{remark}

\begin{example}
Furch's $3$-ball \cite[p.~73]{FUR} \cite[p.~110]{BING} can be triangulated without interior vertices (see e.g. \cite{HACHIweb}). Since it contains a knotted spanning edge, by Proposition \ref{prop:knottednotLC} Furch's ball is not LC.
\end{example}

\begin{remark}\label{remark:hachimori}
In \cite[Lemma 2]{HACHI}, Hachimori claimed that any $3$-ball $C$ obtained 
from a constructible $3$-ball $C'$ via a type~(ii) step is constructible. 
This would imply by Lemma \ref{lem:sudoku} that all LC $3$-balls without interior vertices are constructible, 
which is stronger than Proposition \ref{prop:knottednotLC} since constructible $3$-balls 
do not contain knotted spanning edges \cite[Lemma 1]{HZ}. 
Unfortunately, Hachimori's proof \cite[p. 227]{HACHI} is not satisfactory: 
If $C'=C'_1 \cup C'_2$ is a constructible decomposition of $C'$, and $C_i$ 
is the subcomplex of $C$ with the same facets of $C'_i$, $C=C_1 \cup C_2$ 
need not be a constructible decomposition for $C$. (For example, if the two 
glued triangles both lie on $\partial C_1'$, and if the two vertices that 
the triangles do not have in common lie in $C_1' \cap C_2'$, then $C_1 \cap C_2$ 
is not a $2$-ball and one of $C_1$ and $C_2$ is not a $3$-ball.)

At present we do not know whether Hachimori's claim is true: Does $C'$ admit a different constructible decomposition that survives the type~(ii) step? On this depends the correctness of the algorithm \cite[p.~227]{HACHI} \cite[p.~101]{Hthesis} to test \emph{constructibility} of $3$-balls without interior vertices by cutting them open along triangles with exactly two boundary edges. However, we point out that Hachimori's algorithm can be validly used to decide the \emph{local constructibility} of $3$-balls without interior vertices: In fact, by Lemma \ref{lem:sudoku}, the algorithm proceeds by reversing the LC-construction of the ball.
\end{remark}

\medskip
\noindent
We can now move on to complete the proof of our Theorem \ref{thm:hierarchyballs}.
Inspired by Proposition \ref{prop:knottednotLC}, we show that a \emph{collapsible} $3$-ball without interior vertices may contain a knotted spanning edge. Our construction is a tricky version of Lickorish--Martin's (see Example \ref{thm:examplelick}).

\begin{thm} \label{thm:collapsiblenonLC} Not all collapsible $3$-balls are LC. 
\end{thm}

\begin{proof}
Start with a large $m \times m \times 1$ pile of cubes, triangulated in the standard way, and 
take away two distant cubes, leaving only their bottom squares $X$ and $Y$. The $3$-complex $C$ obtained can be collapsed vertically onto its square basis; 
in particular, it is collapsible, and has no interior vertices.  

Let $C'$ be a $3$-ball with two tubular holes drilled away, but where (1) each hole has been corked at a bottom with a $2$-disk, and (2) the tubes are disjoint but intertwined, so that a closed path that passes through both holes and between these traverses the top resp.\ bottom face of $C'$ yields a trefoil knot (see Figure \ref{fig:LickMar}).
\enlargethispage{3mm}
\begin{figure}[htbf]
\centering\vskip-2mm
 \includegraphics[height=36mm]{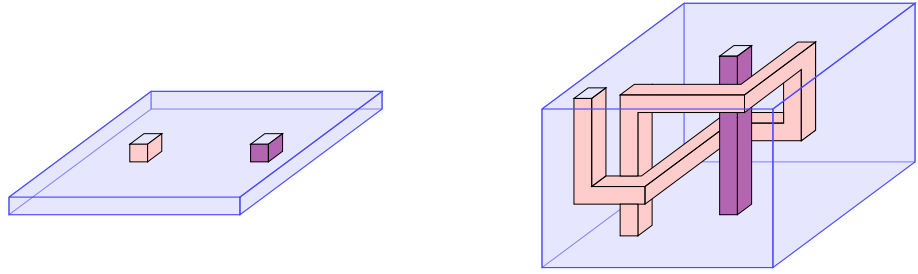}\vskip-3mm
\caption{\small $C$ and $C'$ are obtained from a $3$-ball drilling away two tubular holes, and then ``corking'' the holes on the bottom with $2$-dimensional membranes.} \label{fig:LickMar}
\end{figure}

$C$ and $C'$ are homeomorphic. Any homeomorphism induces on $C'$ a collapsible triangulation with no interior vertices. $X$ and $Y$ correspond via the homeomorphism to the corking membranes of $C'$, which we will call correspondingly $X'$ and~$Y'$. 
To get from $C'$ to a ball with a knotted spanning edge we will carry out two more steps: 
\begin{compactenum}[(i)]
\item create a single edge $[x', y']$ that goes from $X'$ to $Y'$;
\item thicken the ``bottom'' of $C'$ a bit, so that $C'$ becomes a $3$-ball and $[x', y']$ becomes an interior edge 
       (even if its extremes are still on the boundary).
\end{compactenum} 
We perform both steps by adding cones over $2$-disks to the complex.
Such steps preserve collapsibility, but in general they produce interior vertices; thus we choose ``specific'' disks with few interior
vertices. 

\begin{compactenum}[(i)]
\item Provided $m$ is large enough, one finds a ``nice'' strip $F_1,F_2,\dots,F_k$ of triangles on the bottom of $C'$,
such that $F_1\cup F_2\cup \dots\cup F_k$ is a disk without interior vertices, $F_1$ has a single vertex $x'$ in the boundary of $X'$,
while $F_k$ has a single vertex $y'$ in the boundary of $Y'$, and the whole strip intersects $X'\cup Y'$ only in $x'$ and~$y'$. 
Then we add a cone  
to $C'$, setting
\[C_1 \ :=\ C' \cup \left( y'*(F_1\cup F_2\cup \dots\cup F_{k-1}) \right). \] 
(An explicit construction of this type is carried out in \cite[pp.~164-165]{HZ}.)
Thus one obtains a collapsible $3$-complex $C_1$ with no interior vertex, and with a direct edge from $X'$ to $Y'$.
\item Let $R$ be a $2$-ball inside the boundary of $C_1$ that contains in its interior the $2$-complex $X' \cup Y' \cup [x',y']$, and such that every interior vertex of $R$ lies either in $X'$ or in $Y'$. Take a new point $z'$ and define 
$C_2 \  :=\  C_1 \cup (z' * R)$.  
\end{compactenum}
As $z' * R$ collapses onto $R$, it is easy to verify that $C_2$ is a collapsible $3$-ball with a knotted spanning edge $[x', y']$.
By Proposition \ref{prop:knottednotLC}, $C_2$ is not LC.
\end{proof}

\begin{cor}\label{cor:badcollapsible3ball}
There exists a collapsible $3$-ball $B$ such that for any boundary facet $\sigma$, the ball~$B$
does not collapse onto $\partial B - \sigma$.  
\end{cor}

Theorem~\ref{thm:collapsiblenonLC} can be extended to higher dimensions by 
taking cones. In fact, even though the link of an LC complex need not be LC, 
the link of an LC closed star is indeed LC.

 \begin{proposition}\label{prop:LCcones}
 Let $C$ be a $d$-pseudomanifold and $v$ a new point. $C$ is LC if and only 
 if $v*C$ is LC.
 \end{proposition}
 
 \begin{proof}
The implication ``if $C$ is LC, then $v*C$ is LC'' is straightforward.  	

For the converse, assume $T_i$ and $T_{i+1}$ are intermediate steps in the local construction of $v*C$, so that passing from $T_i$ to $T_{i+1}$ we glue together two 
 adjacent $d$-faces $\sigma', \sigma''$ of $\partial T_i$.  Let $F$ be any $(d-1)$-face of $T_i$. If $F$ does not contain $v$, then $F$ is in the boundary of $v * C$, so $F \in \partial T_{i+1}$. Therefore, $F$ cannot 
belong to the intersection of $\sigma'$ and $\sigma''$, which is sunk into the 
interior of $T_{i+1}$.

So, every $(d-1)$-face in the intersection $\sigma' \cap \sigma''$ must contain the vertex $v$. This implies that $\sigma' = v * S'$ and $\sigma'' = v * S''$, with $S'$ and $S''$ \emph{distinct} $(d-1)$-faces. $S'$ and $S''$ must share some $(d-2)$-face, otherwise $\sigma'$ and $\sigma'$ would not be adjacent. So from a local construction of $v * C$ we can read off a local construction of $C$.
 \end{proof} 

\begin{cor} \label{thm:badlycollapsibleDball}
	For every $d\ge3$, not all collapsible $d$-balls are LC.
\end{cor}

\begin{proof}
	All cones are collapsible.
If $B$ is a non-LC $d$-ball, then $v \ast B$ is a non-LC $(d+1)$-ball by Proposition
\ref{prop:LCcones}.
\end{proof}

We conclude this chapter observing that Chillingworth's theorem,
``every geometric triangulation of a convex $3$-dimensional polytope is collapsible'', can be
 strengthened as follows.

\begin{thm}[Chillingworth \cite{CHIL}] \label{thm:chil} Every $3$-ball embeddable as
  a convex subset of the Euclidean $3$-space $\mathbb{R}^3$ is LC.
\end{thm} 

\begin{proof} 
  The argument of Chillingworth for
  collapsibility runs showing that $B \searrow \; \partial B -
  \sigma$, where $\sigma$ is any triangle in the boundary of $B$. Now Theorem \ref{thm:mainballs} ends the proof.
\end{proof}

	Thus any subdivided $3$-simplex is LC. If Hachimori's claim is true (see Remark \ref{remark:hachimori}), then any subdivided $3$-simplex with all vertices on the boundary is also constructible. (So far we can only exclude the presence of knotted spanning edges in it: See Lemma \ref{lem:sudoku}.) However, a subdivided $3$-simplex might be non-shellable even if it has all vertices on the boundary (Rudin's ball is an example).

\section{Upper bounds on the number of LC $d$-spheres.}\label{sec:numbers}

For fixed $d\geq 2$ and a suitable constant $C$ that depends on $d$, 
there are less than $C^N$ combinatorial types of LC $d$-spheres with $N$ facets. Our proof for this fact
is a $d$-dimensional version of the main theorem of Durhuus \& Jonsson \cite{DJ}, and
allows us to determine an explicit constant $C$, for any $d$. It consists of two different phases:
\begin{compactenum}[1. ]
 \item we observe that there are less trees of  
$d$-simplices than planted plane $d$-ary trees, which are counted by order~$d$ Fuss--Catalan numbers;
\item we count the number of ``LC matchings'' according to ridges in the tree of simplices. 
\end{compactenum}

\subsection{Counting the trees of $d$-simplices.}
We will here establish that there are less than $ C_d (N) := \frac{1}{(d-1) N + 1}  \binom{d N}{N} $
trees of $N$ $d$-simplices.  

\begin{lemma} \label{thm:treesA} 
Every tree of $N$ $d$-simplices has 
 $(d-1)N+2$ boundary facets of dimension $d-1$ and
 $N-1$ interior faces of dimension $d-1$.
\\
It has $\frac d2 ((d-1)N+2)$ faces of dimension $d-2$, all of them lying  in the boundary.
\end{lemma}

By \emph{rooted} tree of simplices we mean a tree of simplices $B$ together with a distinguished facet $\delta$ of $\partial B$, whose vertices have been labeled $1, 2, \dots, d$.  
Rooted trees of $d$-simplices are in bijection with 
``planted plane $d$-ary trees'', that is, plane rooted trees such that every non-leaf vertex has
exactly $d$ (left-to-right-ordered) sons; cf.~\cite{Matousek}. 
 
\begin{proposition}\label{prop:countdtrees}
	There is a bijection between rooted trees of $N$ $d$-simplices and planted plane $d$-ary trees with $N$ non-leaf vertices, which in turn are counted by the Fuss--Catalan numbers
	$ C_d (N) = \frac{1}{(d-1) N + 1} \, \binom{d N}{N}$. 
Thus, the number of combinatorially-distinct trees of $N$ $d$-simplices satisfies
\[\frac{1}{(d-1)N+2} \; \frac{1}{d!} \; C_d(N)\ \ \le\ \ \# \; \{ \hbox{ trees of } N \; \hbox{ $d$-simplices } \}\ \ \le\ \ C_d(N). \]
\end{proposition}

\begin{proof}
	Given a rooted tree of $d$-simplices with a distinguished facet $\delta$ in its boundary, there is a unique extension of the labeling of the vertices of $\delta$ to a labeling of all the vertices by labels $1,2,\dots,d+1$, such that no two adjacent vertices get the same label.
	Thus each $d$-simplex receives all $d+1$ labels exactly once. 

	Now, label each $(d-1)$-face by the unique label that none of its vertices has. With this we get an edge-labeled rooted $d$-ary tree 
	whose non-leaf vertices correspond to the $N$ $d$-simplices; the root corresponds to the $d$-simplex that contains $\delta$, 
	and the labeled edges correspond to all the $(d-1)$-faces other than $\delta$. We get a plane tree by 
	ordering the down-edges at each non-leaf vertex left to right  
	according to the label of the corresponding $(d-1)$-face.  
	
	The whole process is easily reversed,   
	so that we can get a rooted tree of $d$-simplices from
	an arbitrary planted plane $d$-ary tree. 
	
	There are exactly
	$ C_d (N) = \frac{1}{(d-1) N + 1} \, \binom{d N}{N}$
	planted plane $d$-ary trees with $N$ interior vertices
	(see e.g.\  
	Aval \cite{Aval}; the integers $C_2(N)$ are the ``Catalan numbers'', which appear in many combinatorial problems, see e.g.\ Stanley \cite[Ex.~6.19]{Stanley2}).  
	Any tree of $N$ $d$-simplices has exactly $(d-1)N+2$ boundary facets, 
	so it can be rooted in exactly $\left((d-1)N + 2 \right)d!$ ways,
	which however need not be inequivalent. This explains the first inequality claimed in the lemma. 
	Finally, combinatorially-inequivalent trees of $d$-simplices also yield inequivalent rooted trees, 
	whence the second inequality follows.
\end{proof}
 
\begin{cor}\label{cor:exponentialbounds}
The number of trees of $N$ $d$-simplices,
for $N$ large, is bounded by 
\[ \binom{dN}{N} \; \sim \; \Big(    d \cdot \big(\tfrac{d}{d-1} \big)^{d-1} \Big)^N \ < \  (d e)^N .\]
\end{cor}

\subsection{Counting the matchings in the boundary.}
We know from the previous section that there are exponentially many trees of $N$ $d$-simplices. Our goal is to find an exponential upper bound for the LC spheres obtainable by a matching of adjacent facets in the boundary of
one fixed tree of simplices. 
 
 \begin{thm} \label{thm:announced} Fix $d \geq 2$. The number of
   combinatorially distinct LC $d$-spheres (or LC $d$-balls) with $N$ facets, for $N$
   large, is not larger than
\[ \Big(  d \cdot \big(\tfrac{d}{d-1} \big)^{d-1} \cdot 2^{\; \frac{2d^2 - d}{3}} \Big)^N .\]
\end{thm} 

\begin{proof} 
Let us fix a tree of $N$ $d$-simplices $B$. We adopt the word ``couple'' to denote a pair of facets in the
boundary of $B$ that are glued to one another
during the local construction of $S$. 

Let us set $D:=\frac12(2+N(d-1))$, which is an integer. By Lemma \ref{thm:treesA}, 
the boundary of the tree of $N$ $d$-simplices contains $2 D$ facets, so each perfect matching is just a set of $D$
pairwise disjoint couples. We are going to partition every perfect matching
into ``rounds''. The first round will contain couples that are
adjacent in the boundary of the tree of simplices. Recursively, the
$(i+1)$-th round will consist of all pairs of facets that
\emph{become} adjacent only after a pair of facets 
   are glued together in the $i$-th round. 

Selecting a pair of adjacent facets is the same as choosing
the ridge between them; and by Lemma \ref{thm:treesA}, the boundary
contains $d D$ ridges. Thus the first round of identifications
consists in choosing $n_1$ ridges out of $d D$, where $n_1$ is some
positive integer. After each identification, at most $d-1$ new ridges
are created; so, after this first round of identifications, there are
at most $(d-1)n_1$ new pairs of adjacent facets.

In the second round, we identify $2n_2$ of these newly adjacent
facets: as before, it is a matter of choosing $n_2$ ridges, out of the
at most $(d-1) n_1$ just created ones. Once this is done, at most
$(d-1) n_2$ ridges are created. And so on.

We proceed this way until all the $2 \, D$ facets in the boundary of  
$B$ have been matched (after $f$ steps, say).
Clearly $n_1 + \ldots + n_f = D$, and since the $n_i$'s are positive
integers, $f \leq D$ must hold. This means there are at most
\[
\sum_{f=1}^{D} \quad \sum_{\begin{array}{c}
    n_1, \ldots, n_f \\
    n_i \geq 1, \,\sum n_i = D \\
    n_{i+1} \leq (d-1)n_i \end{array}} \binom{d D}{n_1}
\binom{(d-1)n_1}{n_2} \binom{(d-1)n_2}{n_3} \cdots \binom{(d-1)
  n_{f-1}}{n_f}
\] 
possible perfect matchings of $(d-1)$-simplices in the
boundary of a tree of $N$ $d$-simplices.

We sharpen this bound by observing that not all ridges
may be chosen in the first round of identifications. For example, we
should exclude those ridges that belong to just two $d$-simplices of $B$. 
An easy double-counting argument reveals
that in a tree of $d$-simplices, the number of ridges belonging to at least 3 $d$-simplices is less than or equal to
$\frac{N}{3} \; \binom{d+1}{2}$. So in the upper bound above we may
replace the first factor $\binom{d D}{n_1}$ with the smaller factor
$\binom{\frac{N}{3} \; \binom{d+1}{2}} {n_1}$.

To bound the sum from above, we use $\binom{n}{k} \leq 2^n$ and  $n_1 + \cdots + n_{f-1}<n_1 + \cdots + n_f=D$,
while ignoring the conditions $n_{i+1} \leq (d-1)n_i$. Thus we obtain
the   upper bound
\[
\hbox{\large 2 }^{\frac{N}{3} \binom{d+1}{2} + \frac{N}{2}(d-1)^2 + (d-1)}  \; \cdot \; \sum_{f=1}^{D} \binom{D -1}{f-1}
 \ \ =\ \ \hbox{\large 2 }^{\frac{N}{3} (2d^2 - d) + (d-1)}.
\] 
The factor $2^{d-1}$ is asymptotically negligible. Thus the number of ways to fold a tree of $N$ $d$-simplices into a sphere via a local construction sequence 
is smaller than $2^{\; \frac{2d^2 - d}{3} \; N}$.
Combining this with Proposition \ref{prop:countdtrees},  
we conclude the proof for the case of $d$-spheres. 
We leave the adaption of the proof for $d$-balls (or general LC $d$-pseudomanifolds) to the reader.
\end{proof}

The upper bound of Theorem~\ref{thm:announced} can be simplified in many ways. For example, for $d \geq 16$ 
it is smaller than $\sqrt[3]{4}^{d^2 N}$.
From Theorem \ref{thm:announced} we obtain explicit upper bounds:
\begin{compactitem} 
\item there are less than $216^N$ LC $3$-spheres with $N$ facets,  
\item there are less than $6117^N$ LC $4$-spheres with $N$ facets,  
\end{compactitem}
and so on.  
We point out that these upper bounds are not sharp, as we overcounted
both on the combinatorial side and on the algebraic side. When $d=2$,
Tutte's upper bound is asymptotically $3.08^N$, whereas the one
given by  our formula is $16^N$. When $d=3$, however, 
our constant is smaller than what follows 
from Durhuus--Jonsson's original argument:  
\begin{compactitem}[ -- ]
\item we improved the matchings-bound from $384^N$ to $32^N$;
\item for the count of trees of tetrahedra we obtain an essentially sharp bound of $6.75^N$.
  (The value implicit in the Durhuus--Jonsson argument \cite[p. 184]{DJ} is  larger since one has to take
  into account that different trees of tetrahedra can have the same unlabeled dual graph.)
\end{compactitem}

\begin{cor}
	For any fixed $d\ge2$, there are exponential lower and upper bounds for the number of LC $d$-spheres on $N$ facets.
\end{cor}

\begin{proof}
	We have just obtained an upper bound; we also get a lower bound from 
	Proposition \ref{prop:countdtrees}/\allowbreak Corollary \ref{cor:exponentialbounds}, since the boundary of a 
	tree of $(d+1)$-simplices is a stacked {$d$-sphere}, and for $d\ge2$ the stacked $d$-sphere determines the
	tree of $(d+1)$-simplices uniquely.	
\end{proof}
We know very little about the number of LC $d$-spheres with $N$ facets when $d$ is not constant and $N$
is relatively small (say, bounded by a polynomial) in terms of $d$ --- and whether the LC condition is 
crucial for that. Compare Kalai \cite{Kalai}.

\medskip

\noindent
\textbf{Acknowledgement.}
We are very grateful to Matthias Staudacher,   
Davide Gabrielli, Niko Witte, Raman Sanyal, Thilo R\"{o}rig, Frank Lutz, Gil Kalai, and Emo Welzl for useful discussions and references.
Many thanks also to the anonymous referees for the very careful reviews.

\begin{small}

\end{small}

\end{document}